\documentclass[10pt,authoryear]{article}%
\usepackage{amsmath}
\usepackage{amsfonts}
\usepackage{mathrsfs}
\usepackage{amssymb,bbm, color}
\usepackage[hidelinks,colorlinks=true,linkcolor=blue,citecolor=blue]{hyperref}
\usepackage{graphicx}
\numberwithin{equation}{section}
\usepackage[body={15cm,20cm}, top=3cm]{geometry}
\usepackage{enumitem} 
\providecommand{\U}[1]{\protect\rule{.1in}{.1in}}
\providecommand{\U}[1]{\protect \rule{.1in}{.1in}}
\newtheorem{theorem}{Theorem}[section]

\newtheorem{corollary}[theorem]{Corollary}

\newtheorem{definition}[theorem]{Definition}

\newtheorem{lemma}[theorem]{Lemma}

\newtheorem{proposition}[theorem]{Proposition}
\newtheorem{assumption}[theorem]{Assumption}
\newtheorem{remark}[theorem]{Remark}

\newenvironment{proof}[1][Proof]{\noindent \textbf{#1.} }{\  \rule{0.5em}{0.5em}}

\def\mR{{\mathbb R}}
\def\Gr{{\operatorname{Gr}}}

 \usepackage{natbib}
\allowdisplaybreaks

\begin{document}
	\title{Solutions and stochastic averaging for delay-path-dependent stochastic variational inequalities in infinite dimensions}
	\author{Ning Ning\thanks{Department of Statistics,
			Texas A\&M University, College Station, Texas, USA. Email: patning@tamu.edu}
	\and Jing Wu \thanks{School of Mathematics, Sun Yat-sen University,
		Guangzhou, China. Email: wujing38@mail.sysu.edu.cn}
		\and Xiaoyan Xu \thanks{School of Mathematics and Statistics, Jiangxi Normal University, Nanchang, China. Email: xuxy@jxnu.edu.cn}}
	\date{}
	\maketitle
	\begin{abstract}
In this paper, we study a very general stochastic variational inequality (SVI) having jumps, random coefficients, delay, and path dependence, in infinite dimensions. Well-posedness in terms of the existence and uniqueness of a  solution is established, and a stochastic averaging principle on strong convergence of a time-explosion SVI to an averaged equation is obtained, both under non-Lipschitz conditions. We illustrate our results on general but concrete examples of finite dimension and infinite dimension respectively, which cover large classes of particle systems with electro-static repulsion, nonlinear stochastic partial differential equations with jumps, semilinear stochastic partial differential equations (especially stochastic reaction–diffusion equations) with delays, and others.
	\end{abstract}
	
	
	\textbf{MSC-classification}: 60H15, 35R60, 49J53, 60G20
	
	
	\section{Introduction}
\label{sec:Introduction}
In this section, we give the background and motivations in Section \ref{sec:Background},  summarize our contributions in Section \ref{sec:contributions}, 
illustrate our modeling strength with examples in Section \ref{sec:Modeling_strength_illustration}, 
followed with the organization of the paper in Section \ref{sec:Organization}.

\subsection{Background and motivations}
\label{sec:Background}
The stochastic variational inequality (SVI) is
a powerful modeling tool for many applied problems in which dynamics,
inequalities, and discontinuities are present, such as
constrained time-dependent physical systems with unilateral constraints,
differential Nash games, and hybrid engineering systems with variable structures, hence 
widely used in economics, finance, optimization, game theory, and others \citep{robinson1979generalized, robinson1982generalized}. 
They are stochastic differential equations (SDEs) with subdifferential operators. They generalize SDEs with reflections at the boundaries of convex domains. 
Wide interest in SVI has been attracted which includes, but is not limited to, the following: \cite{kree1982diffusion} proved the existence and uniqueness of a multivalued differential equation governing deterministic oscillations with bilinear hysteresis; the well-posedness of SVIs was established in \cite{bensoussan1994parabolic} and \cite{bensoussan1997stochastic} for  finite-dimensional spaces and infinite-dimensional spaces, respectively;
a splitting up method was developed in \cite{asiminoaei1997approximation} to investigate SVIs;
the multivalued approach developed in
\cite{cepa1998problame}
clarifies the connection between SVIs and the deterministic Skorohod problem;
\cite{zualinescu2002second} studied a
class of Hamilton-Jacobi-Bellman inequalities by considering an optimal stochastic
control problem with an SVI as the state equation; \cite{ren2012regularity} analyzed the regularity of invariant measures of multivalued
SDEs; \cite{ren2010exponential} investigated the ergodicity property for the transition semigroup of solutions to SVIs; \cite{ren2016approximate} explored approximating continuity and the support theory of general reflected SDEs. 

Stochastic modeling with dependence on history is appropriate and desired in many scientific areas, such as chemistry \citep{brett2013stochastic}, biology \citep{miekisz2011stochastic}, ecology \citep{mao2005stochastic}, and finance \citep{arriojas2007delayed}. To model the  after-effect, time-lag, or time-delay phenomena, stochastic functional differential equations (SFDEs) were introduced and have been playing a significant role in modeling evolutions of dynamical systems where uncertainty affects the current situation not only through the current state but also through the history
(see, e.g. \cite{ruess1996existence,wu1996theory,taniguchi2002existence,caraballo2004attractors, jahanipur2010stochastic, barbu2014existence,bao2016asymptotic}). Given that SDEs in restricted regions have wide applications in finance (see, e.g. \cite{han2016optimal}) and \cite{ning2021well} proved the first well-posedness of path-dependent multi-dimensional SVIs. 

Another factor common in stochastic modeling is jumps, which are widely used in finance, economics, science,   engineering and others, anywhere the underlying randomness containing abrupt changes (see, e.g., \cite{corcuera2005completion}). For example in finance, decisions of the Federal
Reserve and other central banks can cause the stock price to
make a sudden shift, and hence one would like to represent the stock price
by a diffusion process that has jumps (see, e.g. \cite{Bass2004Stochastic} and the references therein). We refer interested readers to \cite{menaldi1985reflected, ren2011multi, wu2012wiener, ren2013optimal, briand2020mean, wang2021reflected, popovic2022large, qian2022reflected} and the references therein for analysis of SVIs and reflected SDEs with jumps.

The study of random media has had a rapid development during at least the last thirty years, whose original motivation came from fields such as condensed matter physics, physical chemistry, geology, and others. It has pervaded almost all fields of probability theory and random models of different kinds, and its typical research approach consists of the inclusion of a random variable in the coefficients of the SDE (see, e.g. \cite{bayraktar2019controlled}). In a pioneer work \citep{kohatsu1997stochastic}, the SDE with the random coefficient of drift and the nonrandom coefficient of diffusion was considered, and it was shown that the coefficient of drift should be linearly bounded with probability one and should satisfy the Lipschitz condition with a random constant. Later, theoretical results of extended models were established, such as 
\cite{alos1999stochastic, hausenblas2007spdes, zubchenko2011properties, kulinich2014strong, makhno2018stochastic}.

In this paper, we aim to introduce jumps and random coefficients to path-dependent SVIs, extending the state space from finite-dimensional spaces to separable Hilbert spaces of finite-dimension or infinite-dimension, and generalizing the history dependence to a new level by incorporating both delay and path-dependence. For such a general SVI, the most crucial question is whether it can be properly defined in terms of the existence and the uniqueness of its solutions. 
Upon establishing the well-posedness, 
the second subject we are concerned with is the averaging principle, which is very efficient when one studies dynamical systems with highly oscillating components. That is, the highly oscillating components can be averaged
to produce an averaged equation, which, compared with the original one, is much easier to analyze over long time scales. The method of averaging principle is handy in capturing the main
behaviors of the dynamical systems while bypassing the challenges caused by the fast-oscillating components of the systems, and thus has been applied successfully in various models (see, e.g., \cite{xu2014averaging, mao2016averaging,dong2018averaging,hu2019strong, mao2019averaging,rockner2021averaging}).

\subsection{Our contributions}
\label{sec:contributions}
Consider a Gelfand triple
$$
\mathbb{V}\subset\mathbb{H}\cong\mathbb{H}^{*}\subset\mathbb{V}^{*},
$$
i.e., $(\mathbb{H}, \langle\cdot,\cdot\rangle_{\mathbb{H}} )$ is a separable Hilbert space, identified with its dual space by the Riesz isomorphism, and $(\mathbb{V},\|\cdot\|_{\mathbb{V}})$ is a Banach space continuously and densely embedded into $\mathbb{H}$. 
Let $\mathbb{K}$ be another separable Hilbert space and $L(\mathbb{K},\mathbb{H})$ denote the space of all bounded linear operators from $\mathbb{K}$ to $\mathbb{H}$. For $Q\in L(\mathbb{K},\mathbb{K})$ being any nonnegative self-adjoint operator, denote $L_{2}^{0}(\mathbb{K},\mathbb{H})$ as the space of all $\Lambda\in L(\mathbb{K},\mathbb{H})$ such that $\Lambda\sqrt{Q}$ is a Hilbert-Schmidt operator and the trace $\text{tr}(\Lambda Q\Lambda^{*})<\infty$, with the norm given by
$\|\Lambda\|_{L_{2}^{0}}^{2}:=\text{tr}(\Lambda Q\Lambda^{*}).
$
Let $\mathbb{U}$ be a separable Hilbert space and $(\mathbb{U},\mathcal{B}(\mathbb{U}))$ be a measurable space with the norm $\|\hspace{0.1mm}\cdot\hspace{0.1mm}\|_{\mathbb{U}}$.
We consider $(\Omega,\{\mathcal{F}_t\}_{t\geq 0},\mathcal{F},\mathbb{P})$ as a complete probability space and  $P(t)$ as a stationary $\{\mathcal{F}_t\}$-measurable Poisson point process valued in $\mathbb{U}$ with L\'evy  measure $\nu$, and $N$ as the Poisson random  measure  induced by $P(t)$, and $\widetilde{N}$ as the compensated Poisson random measure. We consider $W$ as a cylindrical Wiener process with a covariance operator $Q$ on a separable Hilbert space $\mathbb{K}$; see page $341$ of \cite{taniguchi2010existence} for details.

We consider the following delay-path-dependent SVI with jumps and random coefficients:
\begin{align}\label{see}
	\begin{cases}
		dX(t)\in 
		\sigma(t,X_{t})dW(t)+b(t,X_{t})dt+\int_{U}f(t,X_{t},u)\widetilde{N}(dt,du) \\
		\hspace*{1.5cm}+A(t,X(t))dt-\partial\varphi(X(t))dt, \hspace*{2.6cm}t\in [0,T],\\
		X(t)=\phi(t), \hspace*{6.5cm}t\in[-h,0].
	\end{cases}
\end{align}
Here, $\phi: [-h,0]\times\Omega\to \mathbb{V}$ satisfying that $\phi(t)\in\mathcal{F}_0$ for every $t\in[-h,0]$ and  $\phi(\cdot,\omega)$ is c\`{a}dl\`{a}g for every $\omega$, and the path-dependence is modeled through
$$X_{t}(r):=\begin{cases}
	X(t-a(t)),\quad &\mbox{if} ~~r\in[-h, t-a(t)],\\
	X(r),\quad &\mbox{if} ~~r\in[t-a(t), t],\\ 
	X(t),\quad &\mbox{if} ~~r\in[t, T],
\end{cases}$$
where $a$ is a continuous function  mapping  from  $[0,T]$  to $[0,h]$ satisfying $a(0)=h$. If $a(t)\equiv \gamma_0$, then our delayed-path-dependent modeling is a typical delay modeling (see, e.g., \cite{zhou2012reflected, coulibaly2020backward}). If $a(t)=\iota t$ for $0\leq \iota \leq 1$, then our delayed-path-dependent modeling is a path-dependent modeling. In particular, when $a(t)=t$, $X_{t}=X({t\wedge\cdot})$ the path of $X$ up to time $t$, which is a general path-dependent modeling covering various financial mathematics models
(see, e.g., \cite{ning2021well}). The case $a(t)=0$ would erase path-dependence.
The random drift coefficient
$$b:[0,+\infty)\times\mathcal{D}([-h,T];\mathbb{H})\times\Omega\rightarrow\mathbb{H}$$
and the random diffusion coefficient
$$\sigma:[0,+\infty)\times \mathcal{D}([-h,T];\mathbb{H})\times\Omega\rightarrow L^{0}_{2}(\mathbb{K},\mathbb{H})$$ are
progressively measurable, where $\mathcal{D}([-h,T];\mathbb{H})$ denotes the space of all c\`{a}dl\`{a}g functions from $[-h,T]$ to $\mathbb{H}$.  The random function  in the jump component 
$$f:[0,+\infty)\times\mathcal{D}([-h,T];\mathbb{H})\times(\mathbb{U}\backslash\{0\})\times\Omega\rightarrow \mathbb{H}$$ 
is predictably measurable. We consider $A(t,\cdot): \mathbb{V}\to\mathbb{V}^{*}$ as a linear or nonlinear operator. For a convex and lower-semicontinuous  function  $\varphi:\mathbb{H}\rightarrow(-\infty,+\infty]$, the sub-differential operator $\partial\varphi$ is defined as 
\begin{align}
	\label{eqn:subdifferential_def}
	\partial\varphi(x):=\Big\{y\in\mathbb{H};\;\varphi(x)\le\varphi(x')+\langle x-x',y\rangle_{\mathbb{H}},\ \forall x'\in\mathbb{H}\Big\}.
\end{align}

Our first main result (Theorem \ref{wellposed}) is the well-posedness, in terms of the existence and uniqueness of a solution in the following form (see Definition \ref{Def-1} for details):
\begin{align*}
	X(t)=&\phi({0})+\int_{0}^{t}\sigma(s,X_{s})dW(s)+\int_{0}^{t}b(s,X_{s})ds+\int_{0}^{t}A(s,X(s))ds\\
	&+\int_{0}^{t}\int_{U}f(s,X_{s},u)\widetilde{N}(ds,du)-\eta(t),
\end{align*}
where $\eta$  is a process of finite variation. For this general model in infinite dimension, we provide a new general framework to establish its well-posedness with non-Lipschitzian random coefficients,  under conditions that are even weaker than the existing results of less general models \citep{taniguchi2010existence,zhang2007skorohod}.
Recently,  \cite{ning2023Multi} obtained the well-posedness of path-dependent forward-backward SVIs using approximations generated by the Euler approximation approach. Unfortunately, that approach fails for  Eqn.  \eqref{see}. 
Then we first consider a simplified SVI that removes all the path dependence in Eqn.  \eqref{see_omega0}, and obtain the properties of its unique solution in Theorem \ref{wellposed1}. Next, we resort to successive approximation by designing a family of SVIs $\{X^{n}(t)\}_{n\in \mathbb{N}}$ 
in Eqn. \eqref{see7}. By Theorem \ref{wellposed1}, under Assumption \ref{assumption}, there exists a pathwise unique solution $(X^{n}(t),\eta^{n}(t))$ of  Eqn.  \eqref{see7}. At last, in Theorem \ref{wellposed}, we show that $\{X^{n}(t)\}_{n\in \mathbb{N}}$ is  Cauchy in the sense of uniform convergence in probability in $\mathcal{D}([-h,T];\mathbb{H})\cap L^{2}([-h,T];\mathbb{V})$ and the limit process is the unique solution of Eqn. \eqref{see}. 

Our second main result (Theorem \ref{thm:averaging}) is the averaging principle which says that the following delay-path-dependent SVI: 
\begin{equation}\label{eqn:everage_p}
	\begin{cases}
		dX^{\epsilon}(t)\in
		\sigma\Big(\frac{t}{\epsilon},X^{\epsilon}_{t}\Big)dW(t)+ b\Big(\frac{t}{\epsilon},X^{\epsilon}_{t}\Big)dt+\int_{U}f\Big(\frac{t}{\epsilon},X^{\epsilon}_{t},u\Big)\widetilde{N}(dt,du)\\
		\hspace*{1.5cm}+A\Big(X^{\epsilon}(t)\Big)dt-\partial\varphi(X^{\epsilon}(t))dt,\hspace*{2.6cm}t\in[0,T],\\
		X^{\epsilon}(t)=\phi(t),\hspace*{6.5cm} t\in[-h,0],
	\end{cases}
\end{equation}
will converge to the averaged equation below:
\begin{equation}\label{05}
	\begin{cases}
		d\overline{X}(t)\in
		\overline{\sigma}(\overline{X}_{t})dW(t)+\overline{b}(\overline{X}_{t})dt+\int_{U}\overline{f}(\overline{X}_{t},u)\widetilde{N}(dt,du)\\
		\hspace*{1.5cm}+{A}(\overline{X}(t))dt-\partial\varphi(\overline{X}(t))dt,\hspace*{2.85cm} t\in[0,T],\\
		\overline{X}(t)=\phi(t), \hspace*{6.5cm} t\in[-h,0],
	\end{cases}
\end{equation}
as the parameter $\epsilon\in(0,1)$ goes to zero, where $\overline{b}, ~\overline{\sigma}, ~\overline{f}$ are time-averaged functions of $b, ~\sigma, ~f$ given in Section \ref{sec:main_results_stochasticaveragingprinciple}. This is the first time that the averaging principle is established on path-dependent SVIs. 
Although the averaging principle has attracted intensive attention, there are very limited results on general SVIs. Very recently, \cite{chen2022averaging} proved the averaging principle for general SVIs with non-Lipschitz coefficients. But their model is not path-dependent, does not have jumps and stochastic coefficients, and is of finite dimension. Here, we show that $X^{\epsilon}$ can be approximated in the sense of convergence in probability and in the mean square, under  non-Lipschitz assumptions that generalize but weaker than those enforced in  \cite{mao2019averaging,mao2016averaging,xu2014averaging}. Moreover, a challenge that is evident in existing literature, revolves around the oversight of time dependence in the state variable when applying Lipschitz-like assumptions. To address this issue within our framework, we partition the time interval $[0,T]$ into intervals of equal length $\Delta=\epsilon\theta(\epsilon)$, where $\theta(\epsilon)\to\infty$ and $\Delta\to0$ as $\epsilon\to0$. We meticulously manage the time dependence within each interval $[i\Delta, (i+1)\Delta)$. Through the strong convergence established in Theorem \ref{thm:averaging}, supported by the moment estimates outlined in Theorem \ref{thm:moment_estimates}, we are able to conduct a comprehensive analysis of the averaged equation \eqref{05}. 

\subsection{Modeling strength illustration}
\label{sec:Modeling_strength_illustration}
The state space of our model  \eqref{see} is a  separable Hilbert space which is very general. In Section \ref{sec:examples}, we provide two general but concrete examples of  finite-dimensional and infinite-dimensional spaces, respectively. 

In the first example provided in Section \ref{sec:examples1}, we set $\mathbb{V}=\mathbb{H}=\mathbb{V}^{*}=\mathbb{R}^{d}$ and consider the delay-path-dependent SVI in the form of Eqn. \eqref{eqn:example1}. These two specifications simplify our general well-posedness result to Theorem \ref{thm:wellposed_example1}. Now, we provide an important example to visualize one of the modeling strengths of the subdifferential operator. For $x=(x^1,\cdots,x^d)\in\mR^d$ where $d\geq2$, and for $\lambda>0$, let
\begin{align}
	\label{eqn:special_form}
	\varphi(x)=\begin{cases}
		-\lambda\sum_{1\leq i<j\leq d}\ln(x^j-x^i), \quad& ~~x^1<x^2<...<x^d,\\
		+\infty, & ~~\text{otherwise}.
	\end{cases}
\end{align}
Thus $\varphi$ is a convex, proper, lower semicontinuous function with $$D(\partial\varphi)=\Big\{x\in\mR^d;\, x^1<x^2<\cdots<x^d\Big\}.$$ 
Then, as a specical case of Eqn. \eqref{eqn:example1}, the following finite-dimensional particle system with electro-static repulsion:
\begin{align*}
	dX^i(t)={b}(X_t^i)+{\sigma}(X_t^i)d W(t)+\int_U{f}(X_t^i,u)\widetilde{N}(dt,du)+\lambda\sum_{1\leq i<j\leq d}\frac{dt}{X^j(t)-X^i(t)},
\end{align*}
is the generalization with path-dependence and jumps (whose well-posedness is unknown) of interacting diffusing particle systems that have been extensively studied in many classical works (for instances, \cite{sznitman1991topics}, \cite{rogers1993interacting} and references therein). In particular, \cite{cepa1997diffusing} applied subdifferential operators and SVIs to describe the particle system with possible collisions between particles. Our extensions are necessary for interacting particle systems with electro-static repulsion as empirically verified in \cite{nieh2005spontaneously,stylianopoulos2010diffusion,poater2022path}. 

It worths mentioning that the subdifferential $\partial \varphi:\mathbb{R}\to 2^{\mathbb{R}}$ is not only multivalued, but also maximal monotone, and conversely, every multivalued maximal monotone operator on $\mathbb{R}$ is of the subdifferential form of a lower semicontinuous function on $\mathbb{R}$ (\cite{barbu2016stochastic}). In this respect, the subdifferential operator is closely connected to the porous medium operator in porous media equations. The porous media equation is a versatile mathematical model that captures various physical phenomena, such as the flow of an ideal gas and the diffusion of a compressible fluid through porous media. It also finds applications in studying thermal propagation in plasma and plasma radiation. Additionally, the porous media equation plays a significant role in modeling self-organized criticality processes, commonly known as the ``sand-pile model'' or the ``Bak-Tang-Wiesenfeld model''.
Stochastic perturbations of the porous media equation have been extensively utilized by physicists. Mathematicians have dedicated substantial efforts to rigorously investigate the existence of these equations, leading to valuable mathematical results. We refer to \cite{barbu2016stochastic} for theoretical analysis of stochastic porous media equations (SPMEs), while \cite{al2023effective} offers insights into physical science particularly regarding path-dependent features. \cite{LIU2014158} extended the classical results of SPMEs to multivalued SPMEs with jumps, by introducing a multivalued maximal monotone operator from $\mathbb{V}$ to $\mathbb{V}^*$.  

In the second example provided in Section \ref{sec:examples2}, we consider infinite-dimensional spaces
$$\mathbb{V}=W_{0}^{1,2}(O)\subset\mathbb{H}=L^{2}(O)\subset W^{-1,2}(O)=\mathbb{V}^{*},$$
where $O$ is an open bounded domain in $\mathbb{R}^{d}$, $W^{k,2}(O)$ is the standard Sobolev space on $O$,
and $W_{0}^{k,2}$ is the closure of $C_0^{\infty}(O)$ (the set of all infinitely differentiable real-valued functions on $O$ with compact support) with respect to the Sobolev norm on $O$. We consider a general nonlinear stochastic partial differential equation (SPDE) in Eqn. \eqref{eqn:SPDE}.
This example generalizes the SPDE with jumps in Example 2.3 of \cite{zhen2011stochastic}, and the multivalued SPDE discussed in Section 6 of \cite{zhang2007skorohod}. In addition,  this example includes the classical stochastic reaction-diffusion equations and serves as their generalization to nonlinear SPDEs with path-dependence and jumps. A simplified well-posedness result of Theorem \ref{wellposed} is provided in Theorem \ref{thm:wellposed_example2}, where its proof is dedicated to proving the reduced assumptions.

\subsection{Organization of the paper}
\label{sec:Organization}
The rest of the paper proceeds as follows. In Section \ref{sec:main_results}, we present our main results: the existence and uniqueness of a strong solution of the SVI (\ref{see})  and an averaging principle for the SVI \eqref{eqn:everage_p}. In Section \ref{sec:examples}, we illustrate our results using two general but concrete examples of finite and infinite dimension in Sections \ref{sec:examples1} and \ref{sec:examples2}, respectively. Proofs of Section \ref{sec:main_results_wellposed} are provided in the Appendix. Throughout the paper, the letter $C$, with or without subscripts, will denote a positive constant that may change in different occasions.

\section{Main results}	
\label{sec:main_results}
This section covers the main theoretical results of this paper, starting with some preliminaries in Section \ref{sec:main_results_Preliminaries}, then the well-posedness of a strong solution 
in Section \ref{sec:main_results_wellposed}, 
and ending with the stochastic averaging principle in Section \ref{sec:main_results_stochasticaveragingprinciple}. 	

\subsection{Preliminaries}
\label{sec:main_results_Preliminaries}

By the definition of the sub-differential operator $\partial\varphi$ provided in Eqn. \eqref{eqn:subdifferential_def}, it is a set-valued operator and its domain is defined as 
$$D(\partial\varphi):=\big\{x\in\mathbb{H};\;\partial\varphi(x)\neq\emptyset\big\}.$$
It is characterized by its graph
\begin{align*} \Gr(\partial\varphi):=\big\{(x,y)\in\mathbb{H}\times\mathbb{H};\; x\in D(\partial\varphi),\,y\in\partial\varphi(x)\big\}.
\end{align*}
It is multivalued, monotone and maximal monotone in the sense that for any $(x_{1},y_{1})\in \mathbb{H}\times\mathbb{H}$, one has $(x_{1},y_{1})\in \Gr(\partial\varphi)$ if
\begin{align*}
	\langle x_{1}-x_{2},y_{1}-y_{2}\rangle_{\mathbb{H}}\geq 0,\qquad\forall (x_{2},y_{2})\in \Gr(\partial\varphi).
\end{align*}
It is common to assume that $0\in\operatorname{Int}(D(\partial\varphi))$ and $\varphi(x)\geq\varphi(0)=0$ for any $x\in\mathbb{H}$, which will be the case of this paper. The following lemma from \cite{barbu1976nonlinear} provides important properties of $\partial\varphi$.
\begin{lemma} Denote the interior of $D(\partial\varphi)$ as $\operatorname{Int}(D(\partial\varphi))$ and denote the closure of $D(\partial\varphi)$ as $\overline{D(\partial\varphi)}$. The following properties hold:
	\begin{enumerate}
		\item $\operatorname{Int}(D(\partial\varphi))$ and $\overline{D(\partial\varphi)}$ are both convex subsets of $\mathbb{H}$.
		\item For each $x\in D(\partial\varphi)$, $\partial\varphi(x)$ is a closed and convex subset of $\mathbb{H}$.
		\item If $x_{0}\in\operatorname{Int}(D(\partial\varphi))$, then there exists a neighborhood $V$ of $x_{0}$ such that $\big\{\partial\varphi(x);\; x\in V\cap D(\partial\varphi)\big\}$ is a bounded subset of $\mathbb{H}$.
	\end{enumerate}
\end{lemma}

The dualization between $\mathbb{V}$ and its dual space $\mathbb{V}^{*}$ is denoted by ${}_{\mathbb{V}}\langle\cdot,\cdot\rangle_{\mathbb{V}^{*}}$. For 
$(\mathbb{S},\|\hspace{0.1mm}\cdot\hspace{0.1mm}\|_{\mathbb{S}})$ being a normed space and for $p\geq1$, we denote  $\mathcal{M}^{p}([0,T]\times\Omega\,;\,\mathbb{S})$ as the space of progressively measurable functions $h:[0,\infty)\times\Omega\rightarrow\mathbb{S}$ satisfying that 
\begin{align*}
	\mathbb{E}\int_{0}^{T}\|h(s,\omega)\|_{\mathbb{S}}^{p}ds<\infty.
\end{align*}
Let  $\mathcal{M}^{p,\nu}([0,\infty)\times\mathbb{U}\times\Omega\,;\,\mathbb{H})$ denote the space of predictably measurable functions $f:[0,\infty)\times(\mathbb{U}\backslash\{0\})\times\Omega\rightarrow\mathbb{H}$ such that, for $U\in\mathcal{B}(\mathbb{U}\backslash\{0\})$,
\begin{align*}
	\mathbb{E}\int_{0}^{T}\int_{U}\|f(s,u,\omega)\|_{\mathbb{H}}^{p}\nu(du)ds<\infty.
\end{align*}

To estimate the jump component, we will need the Burkh\"older inequality of the following type, with the case that $p\geq 2$ given in \cite{cao2005successive} and the case that $p=1$ presented in  \cite{hausenblas2007spdes}.
\begin{lemma}\label{lem2}
	Let $p=1$ or $2$. Assume that the function $f\in \mathcal{M}^{p,\nu}([0,t]\times\Omega\times\mathbb{U};\mathbb{H})$. Then
	\begin{align*}
		\mathbb{E}\sup_{0\le s\le t}\left\|\int_{0}^{s}\int_{U}f(r,u)\widetilde{N}(dr,du)\right\|_{\mathbb{H}}^{p}\leq C\mathbb{E}\int_{0}^{t}\int_{U}\|f(s,u)\|_{\mathbb{H}}^{p}\nu(du)ds.	
	\end{align*}
\end{lemma}

We will also need the following extension of the Bihari’s inequality, established in Lemma 2.2 of \cite{cao2007successive}.
\begin{lemma}\label{lem3}
	Fix $T>0$. Let $g, h, \lambda$ be non-negative functions on $[0,T]$, $V$ be a continuous non-decreasing function on $[0,T]$, and $\Phi:\mathbb{R}^{+}\rightarrow\mathbb{R}^{+}$ be a continuous non-decreasing function. Set $J(z)=\int_{z_{0}}^{z}\frac{1}{\Phi(r)}dr$, where $z_{0}>0$. If for any $t\in [0,T]$,
	\begin{align}\label{bilem}
		g(t)\leq h(t)+\int_{t}^{T}\lambda(s)\Phi(g(s))dV(s),
	\end{align}
	then $g(t)\leq J^{-1}\Big(J(h^*(t))+\int_{t}^{T}\lambda(s)dV(s)\Big)$, where $h^*(t)=\sup_{t\leq s\leq T}|h(s)|$.
	
	In particular, if in Eqn. \eqref{bilem} $h\equiv0, ~dV(s)=ds$, and $\Phi(0)=0$, $\int_{0+}\frac{1}{\Phi(x)}dx=+\infty$. Then $g\equiv0$.
\end{lemma}

\subsection{Existence and uniqueness}
\label{sec:main_results_wellposed}
In this section, we give our first main result (Theorem \ref{wellposed}), the strong well-posedness of the delay-path-dependent SVI \eqref{see} recalled here as follows:
\begin{align*}
	\begin{cases}
		dX(t)\in 
		\sigma(t,X_{t})dW(t)+b(t,X_{t})dt+\int_{U}f(t,X_{t},u)\widetilde{N}(dt,du) \\
		\hspace*{1.5cm}+A(t,X(t))dt-\partial\varphi(X(t))dt, \hspace*{2.6cm}t\in[0,T],\\
		X(t)=\phi(t), \hspace*{6.5cm}t\in[-h,0],
	\end{cases}
\end{align*}
where, with $a$ being a continuous function from $[0,T]$ to $[0,h]$ satisfying $a(0)=h$, and for $t\in[0,T]$, 
$$X_t=\left\{X\left(\Big(\big(t-a(t)\big)\vee s\Big)\wedge t\right), s\in[-h,T]\right\}=\begin{cases}
	X(t-a(t)),&\mbox{if} ~~s\in[-h, t-a(t)],\\
	X(s),&\mbox{if} ~~s\in[t-a(t), t],\\ 
	X(t), &\mbox{if} ~~s\in[t, T],
\end{cases}$$
$b:\mathbb{R}^+\times\mathcal{D}([-h,T];\mathbb{H})\times\Omega\rightarrow\mathbb{H}$, 
$\sigma:\mathbb{R}^+\times \mathcal{D}([-h,T];\mathbb{H})\times\Omega\rightarrow L^{0}_{2}(\mathbb{K},\mathbb{H})$ are
progressively measurable, and 
$f:\mathbb{R}^+\times\mathcal{D}([-h,T];\mathbb{H})\times(\mathbb{U}\backslash\{0\})\times\Omega\rightarrow \mathbb{H}$ is predictable.


The following assumption will be enforced to obtain the first main result (Theorem \ref{wellposed}), regarding the well-posedness of Eqn. \eqref{see} on a given probability space $(\Omega, \mathcal{F}, \{\mathcal{F}_t; t\geq0\},\mathbb{P})$.
\begin{assumption}
	\label{assumption}
	Suppose the following conditions hold:
	\begin{itemize}
		\item[(H1)] $A(t,\cdot): \mathbb{V}\to\mathbb{V}^*$ is monotone, bounded and hemicontinuous, i.e., for any $u, v\in\mathbb{V}$ and any $t\geq0$, $\mathbb{R}\ni \lambda\to {}_{\mathbb{V}}\big\langle v, A(t,u+\lambda v)\big\rangle_{\mathbb{V}^{*}}$ is continuous. There exists  $m_{0}>0$ and $\beta>0$ such that for any $t\geq 0$ and $v,v'\in \mathbb{V}$,
		\begin{align*}
			m_0\|v-v'\|_{\mathbb{V}}^{2}+2{}_{\mathbb{V}}\big\langle v-v' ,A(t,v)-A(t,v')\big\rangle_{\mathbb{V}^{*}}
			\leq\beta\|v-v'\|_{\mathbb{H}}^{2}.
		\end{align*}
		
		\item[(H2)]  There exists 
		$H(t,z): [0,\infty) \times [0,\infty) \rightarrow \mathbb{R}^{+}$ that is locally integrable in $t$ for any fixed $z$, and  continuous and nondecreasing in $z$ for any $t$, such that for any $\xi \in L^{4}(\Omega \,;\, \mathcal{D}([-h,T]\,;\, \mathbb{H}))$,
		\begin{align*}
			&\mathbb{E}\|\sigma(t, \xi)\|_{L^{0}_{2}}^{4}+\mathbb{E}\|b(t, \xi)\|_{\mathbb{H}}^{4}+\mathbb{E}\left(\int_{U}\|f(t, \xi, u)\|_{\mathbb{H}}^{2} \nu(d u)\right)^{2} +\mathbb{E} \int_{U}\|f(t, \xi, u)\|_{\mathbb{H}}^{4} \nu(d u) \\
			&\hspace{8cm}\leq H\left(t, \mathbb{E}\|\xi\|_{T}^{4}\right),
		\end{align*}
		where $\|\xi\|_{T}:=\sup _{t \in[-h, T]}\|\xi(t)\|_{\mathbb{H}}$.
		For any constant $c>0$, the equation $$\frac{d u}{d t}=c H(t, u)$$ has a solution on $[0,T]$ for any initial value $u(0)$.
		\smallskip
		
		\item[(H3)] For every $t$, $\omega$, and $u$, the functions $\sigma(t,\cdot,\omega)$, $b(t,\cdot,\omega)$ and $f(t,\cdot,u,\omega)$ are continuous. For any $R>0$ and $\xi,\xi'\in L^{2}(\Omega;\mathcal{D}([-h,T];\mathbb{H}))$ satisfying that, for any $t\in[0,T]$ and 
		$$\|\xi\|_{T}\vee\|\xi'\|_{T}\leq R, \qquad a.s.,$$
		the following inequality is satisfied:
		\begin{align*} &\mathbb{E}\|\sigma(t,\xi)-\sigma(t,\xi')\|_{L^{0}_{2}}^{2}+\mathbb{E}\|b(t,\xi)-b(t,\xi')\|_{\mathbb{H}}^{2}\\
			&\hspace{0.9cm}+\mathbb{E}\Bigg[\int_{U}\|f(t,\xi,u)-f(t,\xi',u)\|^{2}_{\mathbb{H}}\nu(du)+
			\Big(\int_U\|f(t,\xi,u)-f(t,\xi',u)\|_{\mathbb{H}}\nu(du)\Big)^{2}\Bigg]\\
			&\hspace{0.9cm}\leq G_{R}(t,\mathbb{E}\|\xi-\xi'\|_{T}^{2}).
		\end{align*}
		Here, $G_{R}(t,z):[0,\infty)\times[0,\infty)\rightarrow\mathbb{R}^{+}$ is locally integrable in $t$ for any fixed $z$ and is continuous and nondecreasing in $z$ for any fixed $t$. Furthermore, $G_{R}(t,0)=0$ for any $t$.
		If $Z(t)$ is a nonnegative function  satisfying that for all $t\in[0,T]$ where $T>0$ and any constant $c>0$,
		\begin{align*}
			Z(t)\leq c\int_{0}^{t}G_{R}(s,Z(s))ds,
		\end{align*}
		then $Z(t)\equiv0$ for all $t\in[0,T]$.
		\smallskip
		
		\item[(H4)] $0\in\operatorname{Int}(D(\partial \varphi))$ and $\varphi(u) \leq C\left(1+\| u \|_{\mathbb{H}}^{l}\right)$ for some $1 \leq l \leq 2$, and $0=\varphi(0) \leq \varphi(u)$ for any $ u \in \mathbb{H}$. Additionally, $\mathbb{E}\sup_{t\in[-h,0]}\|\phi(t)\|_{\mathbb{H}}^4+\mathbb{E}\int_{-h}^0\|\phi(t)\|_{\mathbb{V}}^2dt<\infty$.
		
	\end{itemize}
\end{assumption}

The above assumption generalizes but is weaker than those enforced in \cite{mao2019averaging,mao2016averaging,xu2014averaging}. In the following remark, we provide a concrete example for understanding the above assumption.
\begin{remark}
	Let  $G_{R}(t,z)=\lambda_{0}(t)\sum_{i=1}^{3}\phi_{i}(z)$, where $\lambda_{0}$ is a nonnegative and locally integrable function, and $\phi_{i}$ for $i=1,2,3$ are concave, continuous, non-decreasing functions on $[0,\infty)$ such that $\phi_{i}(0)=0$ and $\int_{0^{+}}\frac{1}{\sum_{i=1}^{k}\phi_{i}(z)}dz=+\infty$ where $0^{+}$ stands for a small neighborhood of $0$ in $\mathbb{R}^{+}$. Then $G_{R}$ satisfies (H3). An example of $\{\phi_{i}\}_{i=1,2,3}$  \citep{cao2007successive,taniguchi2010existence} can be seen below: let $\delta\in(0,1)$ and $\alpha_{0}\in(0,1]$,
	\begin{align*}
		&\phi_{1}(z)=z, \hspace{7.9cm} z\geq 0,\vspace{0.21cm}\\
		&\phi_{2}(z)=\begin{cases}
			0, & z=0,\vspace{0.2cm}\\
			z\Big(\log\frac{1}{z}\Big)^{\alpha_{0}}, & 0<z<\delta,\vspace{0.2cm}\\
			\Big(\log\frac{1}{\delta}\Big)^{\alpha_{0}-1}\Big(\log\frac{1}{\delta}-\alpha_{0}\Big)z+\alpha_{0}\delta\Big(\log\frac{1}{\delta}\Big)^{\alpha_{0}-1}, \hspace{0.8cm}& z\geq\delta,\vspace{0.2cm}\ \ 
		\end{cases}\\
		&\phi_{3}(z)=\begin{cases}
			0, & z=0,\vspace{0.2cm}\\
			z\log\frac{1}{z}\log\log\frac{1}{z}, & 0< z<\delta,\vspace{0.2cm}\\
			\Big(\log\frac{1}{\delta}\log\log\frac{1}{\delta}-\log\log\frac{1}{\delta}-1\Big)z+\delta\log\log\frac{1}{\delta}+\delta, \quad& z\geq\delta.
		\end{cases}
	\end{align*}	
\end{remark}

		Before we proceed to establish the well-posedness of the SVI (\ref{see}), we have to firstly define what a strong solution means here.
		\begin{definition}\label{Def-1}
			A couple of c\`{a}dl\`{a}g processes $(X,\eta)$ is called a solution of  Eqn.   (\ref{see}), if
			\begin{enumerate}
				\item $X\in L^{2}([-h,T]\times\Omega;\mathbb{V})\cap L^{2}\left(\Omega;\mathcal{D}([-h,T];\mathbb{H})\right)$, and for every $t\in[-h,T]$, $X(t)\in\overline{D(\partial\varphi)})$ a.s.;
				\smallskip
				
				\item $\eta(\cdot,\omega)\in\mathcal{D}([0,T];\mathbb{H})$ where $\eta$ is of finite variation and $\eta(0)=0$, a.s.;
				\smallskip
				
				\item $X(t)=\phi(t)$ for $t\in[-h,0]$, and for any $t\in[0,T]$, for any $v\in\mathbb{V}$, the following equality holds a.s.:
				\begin{align*}
					{}_{\mathbb{V}}\langle v,X(t)\rangle_{\mathbb{V}^{*}}=&{}_{\mathbb{V}}\langle v,\;\phi({0})\rangle_{\mathbb{V}^{*}}+{}_{\mathbb{V}}\Big\langle v,\;\int_{0}^{t}\sigma(s,X_{s})dW(s)\Big\rangle_{\mathbb{V}^{*}}\\
					&+{}_{\mathbb{V}}\Big\langle v,\;\int_{0}^{t}b(s,X_{s})ds\rangle_{\mathbb{V}^{*}}
					+{}_{\mathbb{V}}\Big\langle v,\;\int_{0}^{t}A(s,X(s))ds\Big\rangle_{\mathbb{V}^{*}}	\\	&+{}_{\mathbb{V}}\Big\langle v, \;\int_{0}^{t}\int_{U}f(s,X_{s},u)\widetilde{N}(ds,du)\Big\rangle_{\mathbb{V}^{*}}-{}_{\mathbb{V}}\langle v,\eta(t)\rangle_{\mathbb{V}^{*}};
				\end{align*}
				\item For any $\alpha\in\mathcal{D}([0,T];\mathbb{H})$ and any $t\in[0,T]$,
				\begin{align*}
					\langle X(t)-\alpha(t),d\eta(t)\rangle_{\mathbb{H}}\geq (\varphi(X(t))-\varphi(\alpha(t)))dt,\quad a.s..
				\end{align*}
			\end{enumerate}	
		\end{definition}
		
		We first consider the following SVI without path dependence:
		\begin{align}\label{see_omega0}
			\begin{cases}
				dX(t)\in 
				\sigma(t, \omega) d W(t)+b(t, \omega)d t+\int_{U} f(t, u, \omega)\widetilde{N}(ds,du)\\
				\hspace*{1.6cm}+A(t,X(t))dt-\partial\varphi(X(t))dt, \hspace*{2.5cm}t\in [0,T],\\
				X(t)=\phi(t), \hspace*{6.5cm}t\in[-h,0],
			\end{cases}
		\end{align}
		where $\sigma: [0,\infty)\times \Omega \rightarrow L_{2}^{0}(\mathbb{K}, \mathbb{H})$ and $b: [0,\infty) \times \Omega \rightarrow \mathbb{H}$ are progressively measurable,  $f: [0,\infty)\times \Omega \times U \rightarrow \mathbb{H}$ is predictable.
		
		\begin{theorem}\label{wellposed1}
			Assume that Assumption \ref{assumption} holds  and that 
			\begin{align*}
				& \mathbb{E}\int_{0}^{T} \left[\|\sigma(t, \omega)\|_{L_{2}^{0}}^{4}+\|b(t, \omega)\|_{\mathbb{H}}^{4}+\left(\int_{U}\|f(t, u, \omega)\|_{\mathbb{H}}^{2} \nu(d u)\right)^{2}\right]dt\\
				&\hspace{5cm}+\mathbb{E} \int_{0}^{T} \int_{U}\|f(t, u, \omega)\|_{\mathbb{H}}^{4} \nu(du)dt<+\infty.
			\end{align*}
			Then a unique solution $(X(t), \eta(t))$ exists for Eqn. \eqref{see_omega0} satisfying 
			$$
			\mathbb{E}\sup_{t\in[-h,T]}\|X(t)\|_{\mathbb{H}}^4<\infty\quad\text{and}\quad \mathbb{E}\int_{-h}^T\|X(t)\|_{\mathbb{V}}^2dt<\infty.
			$$
			Moreover, the following energy equality holds:
			\begin{align}\label{energy}
				\left\|X(t)\right\|_{\mathbb{H}}^{2}=&\|\phi(0)\|_{\mathbb{H}}^{2}+2 \int_{0}^{t}{}_{\mathbb{V}}\big\langle  X(s), A(s, X(s))\big\rangle_{\mathbb{V}^{*}}d s +2 \int_{0}^{t}\big\langle X(s), b(s, \omega)\big\rangle_{\mathbb{H}} d s\nonumber\\ 
				& +2 \int_{0}^{t}\big\langle X(s), \sigma(s, \omega)d W(s)\big\rangle_{\mathbb{H}} -2 \int_{0}^{t}\big\langle X(s), d \eta(s)\big\rangle_{\mathbb{H}}ds+\int_{0}^{t}\|\sigma(s, \omega)\|_{L_{2}^0}^{2} d s \nonumber\\ 
				& +2 \int_{0}^{t} \int_U\big\langle  X(s), f(s, u, \omega) \big\rangle_{\mathbb{H}} \widetilde{N}(ds,du)+\int_{0}^{t} \int_{U}\|f(s, u, \omega)\|_{\mathbb{H}}^{2} \widetilde{N}(ds,du)\nonumber\\ 
				& +\int_{0}^{t} \int_{U}\|f(s, u, \omega)\|_{\mathbb{H}}^{2} \nu(ds,du).
			\end{align}
		\end{theorem}
		
		\begin{proof}
			See Section \ref{appendix:wellposed1} of the Appendix.
		\end{proof}		
		
		By arguments analogous to that of \cite{zhang2007skorohod}, we have the following estimate.
		\begin{lemma}\label{lem-1}
			Suppose $(X,\eta)$ is a solution of  Eqn.  (\ref{see}). Then there exists some $a_0\in\operatorname{Int}(D(\partial\varphi))$ and constants $m_{1}>0$ and $k_{1}> 0$, such that for all $0\leq s\leq t\leq T$,
			\begin{align*}
				\int_{s}^{t}\langle X(r)-a_0,d\eta(r)\rangle_{\mathbb{H}}\geq m_{1}|\eta|^{t}_{s}-k_{1}\int_{s}^{t}\|X(r)-a_0\|_{\mathbb{H}}dr-k_{1}m_{1}(t-s),
			\end{align*}
			where $|\eta|^{t}_{s}$ is the total variation of $\eta$ on $[s,t]$.
		\end{lemma}

		We establish the existence of a solution $(X,\eta)$ of  Eqn. (\ref{see}) in the sense of Definition \ref{Def-1} through successive approximation. Set 
		$$X^{0}(t)=\left\{
		\begin{array}{ll}
			\phi(t), \quad& \hbox{$t\in[-h,0]$,} \\
			\phi(0), & \hbox{$t\in[0,T]$,}
		\end{array}
		\right.
		$$ and for $n\geq1$,
		\begin{equation}\label{see7}
			\begin{cases}
				dX^{n}(t)\in \sigma(t,X^{n-1}_{t})dW(t)+b(t,X^{n-1}_{t})dt+\int_{U}f(t,X^{n-1}_{t},u)\widetilde{N}(dt,du)\\
				\hspace*{1.5cm}+A(t,X^{n}(t))dt-\partial\varphi(X^{n}(t))dt,\hspace*{2.46cm} t\in[0,T],\\
				X^{n}(t)=\phi(t), \hspace*{6.55cm} t\in[-h,0].
			\end{cases}
		\end{equation}
		Assuming $X^{n-1}$ is well-defined, supposing $$\mathbb{E}\sup_{t\in[-h,T]}\|X^{n-1}(t)\|^4_{\mathbb{H}}<\infty\quad\text{and}\quad \mathbb{E}\int_{-h}^{T}\|X^{n-1}(t)\|^2_{\mathbb{V}}dt<\infty,
		$$
		then by Theorem \ref{wellposed1}, under Assumption \ref{assumption}, there exists a pathwise unique solution $(X^{n}(t),\eta^{n}(t))$ of  Eqn.  \eqref{see7}. According to Definition \ref{Def-1}, $(X^{n}(t),\eta^{n}(t))$ is a couple of c\`{a}dl\`{a}g processes having the following properties:
		\begin{itemize}
			\item $X^n\in L^{2}([-h,T]\times\Omega;\mathbb{V})\cap L^{2}\left(\Omega;\mathcal{D}([-h,T];\mathbb{H})\right)$, and for every $t\in[-h,T]$, $X^n(t)\in\overline{D(\partial\varphi)}$ a.s.;
			\smallskip\vspace{-0.45cm}
			
			\item $\eta^n(\cdot,\omega)\in\mathcal{D}([0,T];\mathbb{H})$ where $\eta^n$ is of finite variation and $\eta^n(0)=0$, a.s.;
			\vspace{0.2cm}
			
			\item For any $t\in[-h,T]$ and any $v\in\mathbb{V}$, 
			\begin{align*}
				{}_{\mathbb{V}}\langle v,X^n(t)\rangle_{\mathbb{V}^{*}}=&{}_{\mathbb{V}}\langle v,\;\phi({0})\rangle_{\mathbb{V}^{*}}+{}_{\mathbb{V}}\Big\langle v,\;\int_{0}^{t}\sigma(s,X^{n-1}_{s})dW(s)\Big\rangle_{\mathbb{V}^{*}}\\
				&+{}_{\mathbb{V}}\Big\langle v,\;\int_{0}^{t}b(s,X^{n-1}_{s})ds\rangle_{\mathbb{V}^{*}}
				+{}_{\mathbb{V}}\Big\langle v,\;\int_{0}^{t}A(s,X^n(s))ds\Big\rangle_{\mathbb{V}^{*}}	\\	&+{}_{\mathbb{V}}\Big\langle v, \;\int_{0}^{t}\int_{U}f(s,X^{n-1}_{s},u)\widetilde{N}(ds,du)\Big\rangle_{\mathbb{V}^{*}}-{}_{\mathbb{V}}\langle v,\eta^n(t)\rangle_{\mathbb{V}^{*}}, \quad a.s.;
			\end{align*}
			\item For any $\alpha\in\mathcal{D}([0,T];\mathbb{H})$ and any $t\in[0,T]$,
			\begin{align*}
				\langle X^n(t)-\alpha(t),d\eta^n(t)\rangle_{\mathbb{H}}\geq (\varphi(X^n(t))-\varphi(\alpha(t)))dt,\quad a.s..
			\end{align*}
		\end{itemize}
		
		The desired properties established in the following proposition are crucial in proving Theorem \ref{wellposed}. 
		
		\begin{proposition}\label{prop1}
			Under Assumption \ref{assumption}, we have
			\begin{align*}
				\sup_{n}\left(\mathbb{E}\sup_{-h\leq t\leq T}\|X^{n}(t)\|_{\mathbb{H}}^{4}+\mathbb{E}\left(\int_{0}^{T}\|X^{n}(t)\|_{\mathbb{V}}^{2}dt\right)^2+\mathbb{E}\left(|\eta^{n}|_{0}^{T}\right)^2\right)<\infty.
			\end{align*}
		\end{proposition}			
		\begin{proof}
			See Section \ref{appendix:prop1} of the Appendix.
		\end{proof}

		The following theorem is the main result of this section. 
		\begin{theorem}
			\label{wellposed}
			Under Assumption \ref{assumption}, there exists a unique solution $(X,\eta)$ of  Eqn.   (\ref{see}) in the sense of Definition \ref{Def-1} satisfying 
			\begin{align*}
				\left(\mathbb{E}\sup_{-h\leq t\leq T}\|X(t)\|_{\mathbb{H}}^{4}+\mathbb{E}\left(\int_{0}^{T}\|X(t)\|_{\mathbb{V}}^{2}dt\right)^2+\mathbb{E}\left(|\eta|_{0}^{T}\right)^2\right)<\infty.
			\end{align*}
			And for any $0\leq s<t\leq T$, there exists a constant $C_T>0$ such that 
			\begin{align*}
				\mathbb{E}\|X(t)-X(s)\|_{\mathbb{H}}^{4}\leq C_T(t-s)\bigg[(t-s)+\int_s^t H(r,c_T)dr\bigg],
			\end{align*}
			where $c_T:=\mathbb{E}\sup_{t\in[-h,T]}\|X(t)\|_{\mathbb{H}}^4$.
		\end{theorem}
		\begin{proof}
			See Section \ref{appendix:wellposed} of the Appendix.
		\end{proof}

		\subsection{A stochastic averaging principle}
		\label{sec:main_results_stochasticaveragingprinciple}
		In this section, we give our second main result (Theorem \ref{thm:averaging}) of the averaging principle that $X^{\epsilon}$ evolving by the delay-path-dependent SVI \eqref{eqn:everage_p} recalled here as follows:
		\begin{equation*}
			\begin{cases}
				dX^{\epsilon}(t)\in
				\sigma\Big(\frac{t}{\epsilon},X^{\epsilon}_{t}\Big)dW(t)+ b\Big(\frac{t}{\epsilon},X^{\epsilon}_{t}\Big)dt+\int_{U}f\Big(\frac{t}{\epsilon},X^{\epsilon}_{t},u\Big)\widetilde{N}(dt,du)\\
				\hspace*{1.5cm}+A\Big(X^{\epsilon}(t)\Big)dt-\partial\varphi(X^{\epsilon}(t))dt,\hspace*{2.6cm}t\in[0,T],\\
				X^{\epsilon}(t)=\phi(t),\hspace*{6.5cm} t\in[-h,0],
			\end{cases}
		\end{equation*}
		converges to $\overline{X}$ evolving by the  averaged equation \eqref{05} recalled here as follows:
		\begin{equation*}
			\begin{cases}
				d\overline{X}(t)\in
				\overline{\sigma}(\overline{X}_{t})dW(t)+\overline{b}(\overline{X}_{t})dt+\int_{U}\overline{f}(\overline{X}_{t},u)\widetilde{N}(dt,du)\\
				\hspace*{1.5cm}+{A}(\overline{X}(t))dt-\partial\varphi(\overline{X}(t))dt,\hspace*{2.85cm} t\in[0,T],\\
				\overline{X}(t)=\phi(t), \hspace*{6.5cm} t\in[-h,0].
			\end{cases}
		\end{equation*}
		To this aim, it is necessary to separate time and the state variable, because of which stronger conditions are assumed as outlined below. Specifically, (B1)-(B3) below represent specific instances of (H1)-(H3) in Assumption \ref{assumption}, and we introduce a new condition (B4). 
		\begin{assumption}
			\label{assumption0}
			Suppose (H4) in Assumption \ref{assumption} and the following conditions hold:
			\begin{itemize}
				\item[(B1)]   Assumption (H1) holds with $A(t,\cdot)=A(\cdot)$ independent of $t$. That is,
				$A: \mathbb{V}\to\mathbb{V}^*$ is monotone, bounded and hemicontinuous and there exists  $m_{0}>0$ and $\beta>0$ such that for any $t\geq 0$ and $v,v'\in \mathbb{V}$,
				\begin{align*}
					m_0\|v-v'\|_{\mathbb{V}}^{2}+2{}_{\mathbb{V}}\big\langle v-v' ,A(v)-A(v')\big\rangle_{\mathbb{V}^{*}}
					\leq\beta\|v-v'\|_{\mathbb{H}}^{2}.
				\end{align*}				
				
				\item[(B2)]  Assumption (H2) holds with
				$$H\Big(t,\mathbb{E}\|\xi\|_{T}^{4}\Big)=k(t)\widetilde{H}\Big(\mathbb{E}\|\xi\|_{T}^{4}\Big),$$
				for any $t\geq 0$ and  $\xi\in L^{4}(\Omega;\mathcal{D}([-h,T];\mathbb{H}))$,  
				where $k$ is a nonnegative integrable function on $[0,\infty)$ satisfying $\lim_{t\to\infty}\frac1t\int_0^tk(s)ds\leq l_0$ for some constant $l_0>0$, and $\widetilde{H}$ is continuous and nondecreasing.
				\smallskip
				
				\item[(B3)] Assumption (H3) holds with
				\begin{align*} G_{R}(t,\mathbb{E}\|\xi-\xi'\|^{2}_{T})= k(t)\widetilde{G}_R(\mathbb{E}\|\xi-\xi'\|_{T}^{2}),
				\end{align*}
				for any $t\in[0,T]$, and random variables $\xi,\xi'$ taking values in $\mathcal{D}([-h,T];\mathbb{H})$ satisfying that almost surely
				$$\|\xi\|_{T}\vee\|\xi'\|_{T}\leq R,$$
				where $k$ is given in (B2) and $\widetilde{G}_R:\mathbb{R}^{+}\rightarrow\mathbb{R}^{+}$ is a  continuous,  nondecreasing function such that $\widetilde{G}_R(0)=0$ and 
				$\int_{0^+}\frac{dz}{z+\widetilde{G}_R(z)}=\infty$.
				\smallskip
				
				\item[(B4)] For any $\zeta\in L^{2}(\Omega;\mathcal{D}([-h,T];\mathbb{H}))$, $v\in\mathbb{V}$,  and $T_1>0$, there exists a positive bounded function $\psi$ satisfying that $\lim_{T_1\rightarrow\infty}\psi(T_1)=0$ and
				\begin{align*}
					&\frac{1}{T_1}\mathbb{E}\int_{0}^{T_1}\|\sigma(s,\zeta)-\overline{\sigma}(\zeta)\|
					_{L_{2}^{0}}^{2}ds\leq\psi(T_1)(1+\mathbb{E}\|\zeta\|_{T}^{2}),\\
					&\frac{1}{T_1^{2}}\mathbb{E}\int_{0}^{T_1}\|b(s,\zeta)-\overline{b}(\zeta)\|_{\mathbb{H}}^{2}
					ds\leq\psi(T_1)(1+\mathbb{E}\|\zeta\|_{T}^{2}),\\
					&\frac{1}{T_1}\mathbb{E}\int_{0}^{T_1}\Big(\int_{U}\|f(s,\zeta,u)-\overline{f}(\zeta,u)\|
					_{\mathbb{H}}\nu(du)\Big)^2ds+\frac{1}{T_1}\mathbb{E}\int_{0}^{T_1}\int_{U}\|f(s,\zeta,u)-\overline{f}(\zeta,u)\|
					_{\mathbb{H}}^{2}\nu(du)ds\\
					&\hspace{4.5cm}\leq\psi(T_1)(1+\mathbb{E}\|\zeta\|_{T}^{2}).
				\end{align*}
			\end{itemize}
		\end{assumption}
		
		By the results in Section \ref{sec:main_results_wellposed}, we obtain the following theorem.		
		\begin{theorem}\label{thm:moment_estimates}
			Under Assumption \ref{assumption0}, there exists a unique solution of  Eqn.   (\ref{eqn:everage_p}) such that
			\begin{align}\label{11}
				\sup_{\epsilon}\mathbb{E}\sup_{t\in[-h,T]}\|X^{\epsilon}(t)\|_{\mathbb{H}}^4<\infty,
			\end{align}
			and there exists a unique solution of  Eqn.   (\ref{05}) such that
			\begin{align}\label{12}
				\mathbb{E}\sup_{-h\leq t\leq T}\|\overline{X}(t)\|_{\mathbb{H}}^{4}<\infty.
			\end{align}
		\end{theorem}
		
		\begin{proof}
			For every $\epsilon>0$, by Theorem \ref{wellposed}, a unique solution $(X^{\epsilon},\eta^{\epsilon})$ exists for  Eqn.   (\ref{eqn:everage_p}). Moreover, following the arguments in Section \ref{appendix:prop1} of the Appendix, we obtain that 
			\begin{align*}
				\mathbb{E}\sup_{t\in[-h,T]}\|X^{\epsilon}(t)\|_{\mathbb{H}}^4\leq 
				C\Big(1+\mathbb{E}\sup_{t\in[-h,0]}\|\phi(t)\|_{\mathbb{H}}^4\Big)+C_T\int_0^Tk\Big(\frac{s}{\epsilon}\Big)\widetilde{H}\Big(\mathbb{E}\sup_{t\in[-h,s]}\|X^{\epsilon}(t)\|_{\mathbb{H}}^4\Big)ds.
			\end{align*}
			Lemma \ref{lem3} yields that 
			\begin{align*}
				\mathbb{E}\sup_{t\in[-h,T]}\|X^{\epsilon}(t)\|_{\mathbb{H}}^4&\leq
				J^{-1}\Bigg(J\Big(C+C\mathbb{E}\sup_{t\in[-h,0]}\|\phi(t)\|_{\mathbb{H}}^4\Big)+\frac{C_T}{\epsilon}\int_0^{T/\epsilon}k(s)ds\Bigg).
			\end{align*}
			Note that  $\lim_{\epsilon\to0}\frac T{\epsilon}\int_0^{T/\epsilon}k(s)ds\leq l_0$,
			it follows that
			\begin{align}\label{Xepsm}
				\sup_{\epsilon}\mathbb{E}\sup_{t\in[-h,T]}\|X^{\epsilon}(t)\|_{\mathbb{H}}^4<\infty.
			\end{align}
			
			We now show that  $\overline{b}$, $\overline{\sigma}$ and $\overline{f}$ satisfy (H2) and (H3) in Assumption \ref{assumption}. 	
			For any random variables $\xi$ valued in $\mathcal{D}([-h,T];\mathbb{H})$,
			\begin{align*}
				\|\overline{b}(\xi)\|_{\mathbb{H}}&\leq \frac1{T_1}\int_0^{T_1}\|b(s,\xi)-\overline{b}(\xi)\|_{\mathbb{H}}ds+\left\|\frac1{T_1}\int_0^{T_1}b(s,\xi)ds\right\|_{\mathbb{H}},
			\end{align*}
			and thus\begin{align*}
				\mathbb{E}\|\overline{b}(\xi)\|_{\mathbb{H}}^4&\leq C\psi(T_1)^2(1+\mathbb{E}\|\xi\|_T^4)+\frac1{T_1}\int_0^{T_1}k(s)\widetilde{H}(\mathbb{E}\|\xi\|_{T}^4)ds\leq C\widetilde{H}(\mathbb{E}\|\xi\|_{T}^4),
			\end{align*}
			where in the last inequality we sent $T_1\to\infty$ and applied (B4). Furthermore, for  $\xi,\xi'$ valued in $\mathcal{D}([-h,T];\mathbb{H})$ with $\|\xi\|_T\vee\|\xi'\|_T\leq R ~a.s.$, 
			\begin{align*}
				\mathbb{E}\|\overline{b}(\xi)-\overline{b}(\xi')\|_{\mathbb{H}}^{2}
				&\leq \frac{3}{T_{1}^{2}}\mathbb{E}\left\|\int_{0}^{T_{1}}\big[b(s,\xi)-\overline{b}(\xi)\big]ds\right\|_{\mathbb{H}}^{2}+\frac{3}{T_{1}}\mathbb{E}\int_{0}^{T_{1}}\|b(s,\xi)-b(s,\xi')\|_{\mathbb{H}}^{2}ds\nonumber\\
				&\quad+\frac{3}{T_{1}^{2}}\mathbb{E}\left\|\int_{0}^{T_{1}}\big[b(s,\xi')-\overline{b}(\xi')\big]ds\right\|_{\mathbb{H}}^{2}\nonumber\\
				&\leq C\psi(T_{1})\Big(1+\mathbb{E}\|\xi\|_{T}^{2}+\mathbb{E}\|\xi'\|_{T}^{2}\Big)+\frac3{T_1}\widetilde{G}_R\Big(\mathbb{E}\|\xi-\xi'\|_{T}^{2}\Big)\int_0^{T_1}k(s)ds.
			\end{align*}
			Sending $T_{1}\to\infty$, we obtain
			\begin{align}\label{barb}
				\mathbb{E}\|\overline{b}(\xi)-\overline{b}(\xi')\|_{\mathbb{H}}^{2}\leq C\widetilde{G}_R\Big(\mathbb{E}\|\xi-\xi'\|_{T}^{2}\Big).
			\end{align}
			Similarly, 
			\begin{align*}
				\mathbb{E}\|\overline{\sigma}(\xi)\|_{\mathbb{H}}^4&\leq C\widetilde{H}(\mathbb{E}\|\xi\|_{T}^4),\\
				\mathbb{E}\Big(\int_{U}\|\overline{f}(\xi,u)\|_{\mathbb{H}}^2\nu(du)\Big)^2
				&+\mathbb{E}\int_{U}\|\overline{f}(\xi,u)\|_{\mathbb{H}}^4\nu(du)\leq C\widetilde{H}(\mathbb{E}\|\xi\|_{T}^4).
			\end{align*}
			Additionally,
			\begin{align*}
				\mathbb{E}\|\overline{\sigma}(\xi)-\overline{\sigma}(\xi')\|_{L_{2}^{0}}^{2}\leq C\widetilde{G}_R\Big(\mathbb{E}\|\xi-\xi'\|_{T}^{2}\Big),\\
				\mathbb{E}\int_{U}\|\overline{f}(\xi,u)-\overline{f}(\xi',u)\|_{\mathbb{H}}^2\nu(du)
				\leq C\widetilde{G}_R\Big(\mathbb{E}\|\xi-\xi'\|_{T}^{2}\Big).
			\end{align*}
			Hence, Assumption \ref{assumption} holds and then there exists 
			a unique solution $(\overline{X},\overline{\eta})$ 
			to  Eqn.   (\ref{05}) by Theorem \ref{wellposed}. Similar to Eqn. \eqref{Xepsm}, we have
			\begin{align*}
				\mathbb{E}\sup_{-h\leq t\leq T}\|\overline{X}(t)\|_{\mathbb{H}}^{4}<\infty\quad\text{and}\quad 
				\mathbb{E}\int_{-h}^T\|\overline{X}(t)\|_{\mathbb{V}}^{2}dt<\infty.
			\end{align*}
		\end{proof}

		
		The following theorem is the main result of this section.
		\begin{theorem}\label{thm:averaging}
			Under Assumption \ref{assumption0}, 
			\begin{align}\label{probability}
				\sup_{0\le s\le T}\|X^{\epsilon}(s)-\overline{X}(s)\|_{\mathbb{H}}\xrightarrow{\mathbb{P}} 0,\qquad \text{as}\ \  \epsilon\rightarrow 0.
			\end{align}
		\end{theorem}
		
		\begin{proof}
			Set 
			$$
			\tau_R:=\inf\Big\{t\geq0; \|X^{\epsilon}( t)\|_{\mathbb{H}}\vee\|\overline{X}\|_{\mathbb{H}}>R\Big\}.
			$$
			It follows from Theorem \ref{thm:moment_estimates} that 
			\begin{align}\label{inprob}
				\lim_{R\to\infty}\sup_{\epsilon}\mathbb{P}\big(\tau_R>T\big)=1.
			\end{align}
			Applying the energy equality \eqref{energy} to $e^{-\beta t}\|X^{\epsilon}( t)-\overline{X}( t)\|_{\mathbb{H}}^{2}$, for $t\leq T\wedge\tau_R$,
			\begin{align*}
				e^{-\beta t}\|X^{\epsilon}(t)-\overline{X}(t)\|_{\mathbb{H}}^{2}
				&=2\int_{0}^{t}e^{-\beta s}\Big\langle X^{\epsilon}(s)-\overline{X}(s),\Big(\sigma\Big(\frac{s}{\epsilon},X^{\epsilon}_{s}\Big)-\overline{\sigma}(\overline{X}_{s})\Big)dW(s)\Big\rangle_{\mathbb{H}}ds\\
				&\quad+2\int_{0}^{t}e^{-\beta s}\Big\langle X^{\epsilon}(s)-\overline{X}(s),b\Big(\frac{s}{\epsilon},X^{\epsilon}_{s}\Big)-\overline{b}(\overline{X}_{s})\Big\rangle_{\mathbb{H}}ds\\
				&\quad+2\int_{0}^{t}e^{-\beta s}{}_{\mathbb{V}}\Big\langle X^{\epsilon}(s)-\overline{X}(s),A\big(X^{\epsilon}(s)\big)-{A}(\overline{X}(s))\Big\rangle_{\mathbb{V}^{*}}ds\\
				&\quad-2\int_{0}^{t}e^{-\beta s}\Big\langle X^{\epsilon}(s)-\overline{X}(s),d\eta^{\epsilon}(s)-d\overline{\eta}(s)\Big\rangle_{\mathbb{H}}\\
				&\quad+2\int_{0}^{t}e^{-\beta s}\int_{U}\Big\langle X^{\epsilon}(s)-\overline{X}(s),f\Big(\frac{s}{\epsilon},X^{\epsilon}_{s},u\Big)-\overline{f}(\overline{X}_{s},u)\Big\rangle_{\mathbb{H}}\widetilde{N}(ds,du)\\
				&\quad+\int_{0}^{t}e^{-\beta s}\Big\|\sigma\Big(\frac{s}{\epsilon},X^{\epsilon}_{s}\Big)-\overline{\sigma}(\overline{X}_{s})\Big\|_{L_{2}^{0}}^{2}ds-\beta \int_{0}^{t}\|X^{\epsilon}(s)-\overline{X}(s)\|_{\mathbb{H}}^{2}ds\\
				&\quad+\int_{0}^{t}e^{-\beta s}\int_{U}\Big\|f\Big(\frac{s}{\epsilon},X^{\epsilon}_{s},u\Big)-\overline{f}(\overline{X}_{s},u)\Big\|_{\mathbb{H}}^{2}\widetilde{N}(ds,du)\\\
				&\quad+\int_{0}^{t}e^{-\beta s}\int_{U}\Big\|f\Big(\frac{s}{\epsilon},X^{\epsilon}_{s},u\Big)-\overline{f}(\overline{X}_{s},u)\Big\|_{\mathbb{H}}^{2}\nu(du)ds.
			\end{align*}	
			Divide $[0,T]$ into intervals of the same length $\Delta=\epsilon\theta(\epsilon)$ satifying that $\theta(\epsilon)\to\infty$ and $\Delta\to0$ as $\epsilon\to0$. For $s\in[i\Delta, (i+1)\Delta)$, denote $s(\Delta):=i\Delta$. 
			We have
			\begin{align*}
				&2\mathbb{E}\int_{0}^{T\wedge\tau_R}e^{-\beta s}\Big\langle X^{\epsilon}(s)-\overline{X}(s),b\Big(\frac{s}{\epsilon},X^{\epsilon}_{s}\Big)-\overline{b}(\overline{X}_{s})\Big\rangle_{\mathbb{H}}ds\\
				&\leq\frac{1}{8}\mathbb{E}\sup_{0\leq s\leq T\wedge\tau_R}\Big(e^{-\beta s}\|X^{\epsilon}(s)-\overline{X}(s)\|_{\mathbb{H}}^{2}\Big)+C\mathbb{E}\int_{0}^{T\wedge\tau_R}e^{-\beta s}\Big\|b\Big(\frac{s}{\epsilon},X^{\epsilon}_{s}\Big)-b\Big(\frac{s}{\epsilon},\overline{X}_s\Big)\Big\|_{\mathbb{H}}^2ds\\
				&\quad+C\mathbb{E}\int_{0}^{T\wedge\tau_R}e^{-\beta s}\Big\|b\Big(\frac{s}{\epsilon},\overline{X}_s\Big)-b\Big(\frac{s}{\epsilon},\overline{X}_{s(\Delta)}\Big)\Big\|_{\mathbb{H}}^2ds\\
				&\quad+C\mathbb{E}\int_{0}^{T\wedge\tau_R}e^{-\beta s}\Big\|b\Big(\frac{s}{\epsilon},\overline{X}_{s(\Delta)}\Big)-\overline{b}(\overline{X}_{s(\Delta)})\Big\|_{\mathbb{H}}^2ds\\
				&\quad+C\mathbb{E}\int_{0}^{T\wedge\tau_R}e^{-\beta s}\big\|\overline{b}(\overline{X}_{s(\Delta)})-\overline{b}(\overline{X}_{s})\big\|_{\mathbb{H}}^{2}ds\\
				&\leq\frac{1}{8}\mathbb{E}\sup_{0\leq s\leq T\wedge\tau_R}\Big(e^{-\beta s}\|X^{\epsilon}(s)-\overline{X}(s)\|_{\mathbb{H}}^{2}\Big)+C\int_{0}^{T\wedge\tau_R}e^{-\beta s}k\Big(\frac{s}{\epsilon}\Big)\widetilde{G}_R\Big(\mathbb{E}\|X_s^{\epsilon}-\overline{X}_s\|_T^2\Big)ds\\
				&\quad+C\int_{0}^{T\wedge\tau_R}e^{-\beta s}k\Big(\frac{s}{\epsilon}\Big)\widetilde{G}_R\Big(\mathbb{E}\|\overline{X}_s-\overline{X}_{s(\Delta)}\|_T^2\Big)ds+C\sum_{i}\mathbb{E}\int_{i\Delta}^{(i+1)\Delta}\Big\|b\Big(\frac{s}{\epsilon},\overline{X}_{i\Delta}\Big)-\overline{b}(\overline{X}_{i\Delta})\Big\|_{\mathbb{H}}^2ds\\
				&\quad+C\int_{0}^{T\wedge\tau_R}e^{-\beta s}\widetilde{G}_R\Big(\mathbb{E}\|\overline{X}_s-\overline{X}_{s(\Delta)}\|_T^2\Big)ds.
			\end{align*}
			By  (B4) in Assumption \ref{assumption0} and Eqn.   (\ref{barb}), 
			\begin{align}\label{barin}
				C\sum_{i}\mathbb{E}\int_{i\Delta}^{(i+1)\Delta}\Big\|b\Big(\frac{s}{\epsilon},\overline{X}_{i\Delta}\Big)-\overline{b}(\overline{X}_{i\Delta})\Big\|_{\mathbb{H}}^2ds
				&\leq C\Delta\sum_{i}\psi(\theta(\epsilon))\Big(1+\mathbb{E}\|\overline{X}\|_T^2\Big)\nonumber\\
				&\leq C\psi(\theta(\epsilon))\Big(1+\mathbb{E}\|\overline{X}\|_T^2\Big).\end{align}
			Applying Eqn. \eqref{12} and the energy equality \eqref{energy}, with arguments similar to the proof of Theorem \ref{wellposed} in Section \ref{appendix:prop1} of the Appendix, we obtain that for any $0\leq s\leq t\leq T$,
			\begin{align*}
				\mathbb{E}\|\overline{X}(t)-\overline{X}(s)\|_{\mathbb{H}}^4\leq C_T(t-s)^2.
			\end{align*}
			Hence, by Kolmogorov's continuity criterion, denoting $\delta:=\Delta+\sup_{|t-s|\leq\Delta}|a(t)-a(s)|$, we have
			\begin{align}\label{conti}
				\mathbb{E}\|\overline{X}_s-\overline{X}_{s(\Delta)}\|_T^4
				&\leq\mathbb{E}\sup_{|t-s|\leq \delta}\|\overline{X}(t)-\overline{X}(s)\|_{\mathbb{H}}^4\leq C\delta^{\alpha}, \qquad\alpha\in(0,1).
			\end{align}
			Therefore, 
			\begin{align*}
				&\hspace{-0.7cm}2\mathbb{E}\int_{0}^{T\wedge\tau_R}e^{-\beta s}\left\langle X^{\epsilon}(s)-\overline{X}(s),\; b\Big(\frac{s}{\epsilon},X^{\epsilon}_{s}\Big)-\overline{b}(\overline{X}_{s})\right\rangle_{\mathbb{H}}ds\\
				&\le\frac{1}{8}\mathbb{E}\sup_{0\le s\le T\wedge\tau_R}e^{-\beta s}\|X^{\epsilon}(s)-\overline{X}(s)\|_{\mathbb{H}}^{2}+C\psi(\theta(\epsilon))(1+\mathbb{E}\|\overline{X}\|_T^2)+C_{T}\widetilde{G}_R(\delta^{\alpha/2})\\
				&\quad+C\int_{0}^{T\wedge\tau_R}e^{-\beta s}k\Big(\frac{s}{\epsilon}\Big)\widetilde{G}_R(\mathbb{E}\|X^{\epsilon}-\overline{X}\|_s^2)ds.
			\end{align*}
			Similar to Eqn. \eqref{barin}, we have  
			\begin{align*}
				&\hspace{-0.7cm}\mathbb{E}\int_0^{T\wedge\tau_R}e^{-\beta s}\Big\|\sigma\Big(\frac{s}{\epsilon},X^{\epsilon}_s\Big)-\overline{\sigma}(\overline{X}_s)\Big\|_{L_2^0}^2ds\\
				&\leq C\psi(\theta(\epsilon))(1+\mathbb{E}\|\overline{X}\|_T^2)+C_{T}\widetilde{G}_R(\delta^{\alpha/2})
				+C\int_{0}^{T\wedge\tau_R}e^{-\beta s}k\Big(\frac{s}{\epsilon}\Big)\widetilde{G}(\mathbb{E}\|X^{\epsilon}-\overline{X}\|_s^2)ds,
			\end{align*}
			and
			\begin{align*}
				&\hspace{-0.7cm}\mathbb{E}\int_{0}^{T\wedge\tau_R}\Big(\int_{U}\Big\|f\Big(\frac{r}{\epsilon},X^{\epsilon}_{r},u\Big)-\overline{f}(\overline{X}_{r},u)\Big\|_{\mathbb{H}}\nu(du)\Big)^2dr\\
				&\leq C\psi(\theta(\epsilon))(1+\mathbb{E}\|\overline{X}\|_T^2)+C_{T}\widetilde{G}_R(\delta^{\alpha/2})
				+C\int_{0}^{T\wedge\tau_R}e^{-\beta s}k\Big(\frac{s}{\epsilon}\Big)\widetilde{G}(\mathbb{E}\|X^{\epsilon}-\overline{X}\|_s^2)ds.
			\end{align*}
			It follows from the Burkh\"older-Davis-Gundy (BDG) inequality and Young's inequality that
			\begin{align*}
				&\hspace{-0.7cm}\mathbb{E}\sup_{0\leq s\leq T\wedge\tau_R}2\Bigg|\int_{0}^{s}e^{-\beta r}\Big\langle X^{\epsilon}(r)-\overline{X}(r),\;\Big(\sigma\Big(\frac{r}{\epsilon},X^{\epsilon}_{r}\Big)-\overline{\sigma}(\overline{X}_{r})\Big)dW_{r}\Big\rangle_{\mathbb{H}}\Bigg|\\
				&\leq\frac{1}{8}\mathbb{E}\sup_{0\le s\le T\wedge\tau_R}e^{-\beta s}\|X^{\epsilon}(s)-\overline{X}(s)\|_{\mathbb{H}}^{2}
				+C\mathbb{E}\int_{0}^{T\wedge\tau_R}e^{-\beta r}\Big\|\sigma\Big(\frac{r}{\epsilon},X^{\epsilon}_{r}\Big)-\overline{\sigma}(\overline{X}_{r})\Big\|_{L_2^0}^2dr\\
				&\leq \frac{1}{8}\mathbb{E}\sup_{0\le s\le T\wedge\tau_R}e^{-\beta s}\|X^{\epsilon}(s)-\overline{X}(s)\|_{\mathbb{H}}^{2}+ C_T\widetilde{G}_R(\delta^{\alpha/2})+C\psi(\theta(\epsilon))(1+\mathbb{E}\|\overline{X}\|_T^2)\\
				&\quad+C\int_0^{T\wedge\tau_R}k\Big(\frac{s}{\epsilon}\Big)\widetilde{G}_R\big(\mathbb{E}\|X^{\epsilon}-\overline{X}\|_s^2\big)ds.
			\end{align*}
			By Lemma \ref{lem2} and (B4) in Assumption \ref{assumption0},
			\begin{align*}
				&\mathbb{E}\sup_{0\leq s\leq T\wedge\tau_R}2\Bigg|\int_{0}^{s}\int_{U}e^{-\beta s}\Big\langle X^{\epsilon}(r)-\overline{X}(r),\; f\Big(\frac{r}{\epsilon},X^{\epsilon}_{r},u\Big)-\overline{f}(\overline{X}_{r},u)\Big\rangle_{\mathbb{H}}\widetilde{N}(dr,du)\Bigg|\\
				&\leq C \mathbb{E}\int_{0}^{T\wedge\tau_R}\int_{U}e^{-\beta r}\|X^{\epsilon}(r)-\overline{X}(r)\|_{\mathbb{H}}\Big\|f\Big(\frac{r}{\epsilon},X^{\epsilon}_{r},u\Big)-\overline{f}(\overline{X}_{r},u)\Big\|_{\mathbb{H}}\nu(du)dr\\
				&\leq\frac{1}{6}\mathbb{E}\sup_{0\leq s\leq T\wedge\tau_R}e^{-\beta s}\|X^{\epsilon}(s)-\overline{X}(s)\|_{\mathbb{H}}^{2}+C{T} \mathbb{E}\int_{0}^{T\wedge\tau_R}\Big(\int_{U}\Big\|f\Big(\frac{r}{\epsilon},X^{\epsilon}_{r},u\Big)-\overline{f}(\overline{X}_{r},u)\Big\|_{\mathbb{H}}\nu(du)\Big)^2dr\\
				&\leq \frac{1}{6}\mathbb{E}\sup_{0\le s\le T\wedge\tau_R}e^{-\beta s}\|X^{\epsilon}(s)-\overline{X}(s)\|_{\mathbb{H}}^{2}+ C_T\widetilde{G}_R(\delta^{\alpha/2})+C\psi(\theta(\epsilon))(1+\mathbb{E}\|\overline{X}\|_T^2)\\
				&\quad+C\int_0^{T\wedge\tau_R}k\Big(\frac{s}{\epsilon}\Big)\widetilde{G}_R\big(\mathbb{E}\|X^{\epsilon}-\overline{X}\|_s^2\big)ds.
			\end{align*}
			By Lemma \ref{lem2} and the Young's inequality, similar to Eqn.  \eqref{barin}, we have
			\begin{align*}
				&\hspace{-0.5cm}\mathbb{E}\sup_{0\leq s\leq T\wedge\tau_R}\int_{0}^{s}\int_{U}e^{-\beta s}\Big\|f\Big(\frac{r}{\epsilon},X^{\epsilon}_{r},u\Big)-\overline{f}(\overline{X}_{r},u)\Big\|_{\mathbb{H}}^{2}\widetilde{N}(dr,du)\\
				&\leq C\int_{0}^{T\wedge\tau_R}\int_{U}e^{-\beta s}\Big\|f\Big(\frac{r}{\epsilon},X^{\epsilon}_{r},u\Big)-\overline{f}(\overline{X}_{r},u)\Big\|_{\mathbb{H}}^{2}\nu(du)dr\\
				&\leq C_T\widetilde{G}_R(\delta^{\alpha/2})+C\psi(\theta(\epsilon))(1+\mathbb{E}\|\overline{X}\|_T^2)+C\int_0^{T\wedge\tau_R}k\Big(\frac{s}{\epsilon}\Big)\widetilde{G}_R\big(\mathbb{E}\|X^{\epsilon}-\overline{X}\|_s^2\big)ds.
			\end{align*}
			Condition (B1) in Assumption \ref{assumption0} yields
			\begin{align*}
				&\hspace{-0.7cm}2\int_{0}^{T\wedge\tau_R}e^{-\beta s}{}_{\mathbb{V}}\Big\langle X^{\epsilon}(s)-\overline{X}(s),A(X^{\epsilon}(s))-{A}(\overline{X}(s))\Big\rangle_{\mathbb{V}^{*}}ds\\
				&\leq 2\beta\int_{0}^{T\wedge\tau_R}e^{-\beta s}\|X^{\epsilon}(s)-\overline{X}(s)\|_{\mathbb{H}}^2ds-2m_0\int_{0}^{T\wedge\tau_R}e^{-\beta s}\|X^{\epsilon}(s)-\overline{X}(s)\|_{\mathbb{V}}^2ds.
			\end{align*}
			The above estimates together with Theorem \ref{thm:moment_estimates} yields
			\begin{align*}
				&\hspace{-0.5cm}\mathbb{E}\sup_{0\leq s\leq T\wedge\tau_R}e^{-\beta s}\|X^{\epsilon}(s)-\overline{X}(s)\|_{\mathbb{H}}^{2}\\
				\leq& \frac{1}{2}\mathbb{E}\sup_{0\leq s\leq T\wedge\tau_R}e^{-\beta s}\|X^{\epsilon}(s)-\overline{X}(s)\|_{\mathbb{H}}^{2}+C_T\widetilde{G}_R(\delta^{\alpha/2})+C\psi(\theta(\epsilon))(1+\mathbb{E}\|\overline{X}\|_T^2)\\
				&+C\int_0^T\Big[\mathbb{E}\sup_{s\in[0,t\wedge\tau_R]}\|X^{\epsilon}(s)-\overline{X}(s)\|_{\mathbb{H}}^2+k\Big(\frac{t}{\epsilon}\Big)\widetilde{G}_R\Big(\mathbb{E}\sup_{s\in[0,t\wedge\tau_R]}\|X^{\epsilon}(s)-\overline{X}(s)\|_{\mathbb{H}}^2\Big)\Big]dt.
			\end{align*}
			Set
			\begin{align}
				\label{eqn:reviewer2_1}
				z(T):=\limsup_{\epsilon\rightarrow0}\mathbb{E}\sup_{0\leq s\leq T\wedge\tau_R}\|X^{\epsilon}(s)-\overline{X}(s)\|_{\mathbb{H}}^{2}.
			\end{align}
			Then 
			\begin{align*}
				z(T)&\leq e^{\beta T}\limsup_{\epsilon\rightarrow0}\mathbb{E}\sup_{0\leq s\leq T\wedge\tau_R}e^{-\beta s}\|X^{\epsilon}(s)-\overline{X}(s)\|_{\mathbb{H}}^{2}\\
				&\leq C_T\widetilde{G}_R(\delta^{\alpha/2})+C\psi(\theta(\epsilon))(1+\mathbb{E}\|\overline{X}\|_T^2)\\
				&\quad+C_T\int_0^T\Big[\mathbb{E}\sup_{s\in[0,t\wedge\tau_R]}\|X^{\epsilon}(s)-\overline{X}(s)\|_{\mathbb{H}}^2+k\Big(\frac{t}{\epsilon}\Big)\widetilde{G}_R\Big(\mathbb{E}\sup_{s\in[0,t\wedge\tau_R]}\|X^{\epsilon}(s)-\overline{X}(s)\|_{\mathbb{H}}^2\Big)\Big]dt.
			\end{align*}
			Noting that $\psi(\theta(\epsilon))\to0$ and $\delta\to0$ as $\epsilon\to0$, and  
			\begin{align}
				\label{eqn:reviewer2_2}
				z(T)\leq C_{T}\int_{0}^{T}\Big(1+k\Big(\frac{t}{\epsilon}\Big)\Big)\big[z(t)+\widetilde{G}_R(z(t))\big]ds,
			\end{align}
			which together with Bihari's inequality in Lemma \ref{lem3} yield
			$
			z(T)\equiv0.
			$
			Moreover by Eqn. \eqref{inprob}, for any $\delta>0$, 
			\begin{align*}
				\mathbb{P}\Big(\sup_{t\in[0,T]}\|X^{\epsilon}(s)-\overline{X}(s)\|_{\mathbb{H}}^{2}>\delta\Big)
				\leq&\mathbb{P}\Big(\sup_{t\in[0,T\wedge\tau_R]}\|X^{\epsilon}(s)-\overline{X}(s)\|_{\mathbb{H}}^{2}>\delta\Big)
				+\mathbb{P}\Big(\tau_R\leq T\Big)\\
				&\to0, \qquad\mbox{as} ~~~\epsilon\to0.
			\end{align*}
		\end{proof}	
		
		\begin{remark}\label{rem}
			Through the change of variables, we can see that $(Y^{\epsilon}(t),K^{\epsilon}(t)):=(X^{\epsilon}(\epsilon t),\eta^{\epsilon}(\epsilon t))$ solves the following equation:
			\begin{equation}\label{Yeps}
				\begin{cases}
					dY^{\epsilon}(t)=
					\sqrt{\epsilon}\sigma(t,Y^{\epsilon}_{t})d\widetilde{W}(t)+ \epsilon b({t},Y^{\epsilon}_{t})dt+\int_{U}f({t},Y^{\epsilon}_{t},u)\widetilde{N}^{\epsilon}(dt,du)\\
					\hspace*{1.5cm}+\epsilon A(Y^{\epsilon}(t))dt- dK^{\epsilon}(t),\hspace*{3.35cm}t\in[0,T],\\
					Y^{\epsilon}(t)=\phi(t),\hspace*{6.5cm} t\in[-h,0],
				\end{cases}
			\end{equation}
			where $\widetilde{N}^{\epsilon}(dt,du)={N}^{\epsilon}(dt,du)-\epsilon\nu(du)dt$ and $N^{\epsilon}$ is the Poisson random measure with intensity $\epsilon \nu$. Moreover, $(\overline{Y}(t),\overline{K}(t)):=(\overline{X}(\epsilon t), \overline{\eta}(\epsilon t))$ solves the following equation:
			\begin{equation}\label{Ybar}
				\begin{cases}
					d\overline{Y}(t)=
					\sqrt{\epsilon}\overline{\sigma}(\overline{Y}_{t})d\widetilde{W}(t)+ \epsilon \overline{b}(\overline{Y}_{t})dt+\int_{U}\overline{f}(\overline{Y}_{t},u)\widetilde{N}^{\epsilon}(dt,du)\\
					\hspace*{1.5cm}+\epsilon  {A}(\overline{Y}(t))dt- d\overline{K}(t),\hspace*{3.5cm}t\in[0,T],\\
					\overline{Y}(t)=\phi(t),\hspace*{6.5cm} t\in[-h,0].
				\end{cases}
			\end{equation}
			The main results in \cite{xu2014averaging}, \cite {mao2016averaging} and \cite{mao2019averaging} are  devoted to multivalued SDEs, SFDEs, and finite-dimensional multivalued SDEs with jumps, respectively.
			They have proved that $\sup_{t\in[0,T]}|Y^{\epsilon}(t)-\overline{Y}(t)|\to0$ in the sense of convergence in probability or in $L^2$.  Clearly, our result in Theorem \ref{thm:averaging} generalizes theirs. 
		\end{remark}

		Furthermore, we have the following corollary.
		\begin{corollary}\label{cor1}
			Suppose $\widetilde{G}_R$ in Assumption \ref{assumption0} does not depend on $R$, i.e., Assumption \ref{assumption0} holds globally. Then  for any $\delta>0$, there exists $\epsilon_0\in(0,1)$ such that
			\begin{align}\label{rate}
				\sup_{\epsilon\in(0,\epsilon_0)}\mathbb{E}\sup_{0\le s\le T}\|X^{\epsilon}(s)-\overline{X}(s)\|_{\mathbb{H}}^{2}<\delta.
			\end{align}
			
		\end{corollary}
		\begin{proof}
			By the proof of Theorem \ref{thm:averaging}, we have
			\begin{align*}
				\mathbb{E}\sup_{0\leq s\leq T}\|X^{\epsilon}(s)-\overline{X}(s)\|_{\mathbb{H}}^{2}\leq C_{T}\psi(\theta(\epsilon))+C_{\lambda,T}\int_{0}^{T}k\Big(\frac{t}{\epsilon}\Big)\widetilde{G}\Big(\mathbb{E}\sup_{0\leq s\leq t}\|X^{\epsilon}(s)-\overline{X}(s)\|_{\mathbb{H}}^{2}\Big)dt.
			\end{align*}
			Let $\widetilde{J}(z)=\int_{z_{0}}^{z}\frac{1}{r+\widetilde{G}(r)}dr$ for some $z_{0}>0$. It follows from Lemma \ref{lem3} that
			\begin{align*}
				\mathbb{E}\sup_{0\leq s\leq T}\|X^{\epsilon}(s)-\overline{X}(s)\|_{\mathbb{H}}^{2}\leq \widetilde{J}^{-1}\Bigg(\widetilde{J}\Big(C_{T}\psi(\theta(\epsilon))\Big)+C_{\lambda,T}\Bigg),
			\end{align*}	
			and then  Eqn.   (\ref{rate}) follows by Assumption \ref{assumption0}.
		\end{proof}

		\section{General examples}
		\label{sec:examples}
		In this section, we illustrate our results on two general but concrete examples of finite dimension in Section \ref{sec:examples1} and infinite dimension in Section \ref{sec:examples2}. The main results obtained in the last section can be simplified correspondingly. The importance of these two examples is provided in Section \ref{sec:contributions}.
		
		\subsection{Finite-dimensional example}
		\label{sec:examples1}
		Set  $\mathbb{V}=\mathbb{H}=\mathbb{V}^{*}=\mathbb{R}^{d}$ and consider the following delay-path-dependent SVI of finite dimension which is a special case of Eqn. \eqref{see}:
		\begin{equation}\label{eqn:example1}
			\begin{split}
				\begin{cases}
					dX(t)\in b_1(t,X(t))dt+
					\int_{-a_1(t)}^0b_2(t,X({t}+s))dsdt+\sigma(t,X({t}-a_2(t)))dW({t})\\
					\hspace{1.53cm}+\int_{U}\int_{-a_3(t)}^0f(t,X(t+s),u)ds\widetilde{N}(dt,du)-\partial\varphi(X(t))dt, \hspace{0.8cm}t\in[0,\infty),\\
					X(t)=\phi(t), \hspace{8.5cm}t\in[-h,0],
				\end{cases}
			\end{split}	
		\end{equation}
		where $a_i: [0,\infty)\to [0,h]$ for $i=1,2,3$ are continuous functions. 		
		Our well-posedness result in Theorem \ref{wellposed} can be simplified as below.	
		\begin{theorem}
			\label{thm:wellposed_example1}
			Assume $\varphi$ satisfies (H4) in Assumption \ref{assumption} and the following conditions hold:
			\begin{enumerate}
				\item $b_1$ is bounded and there exist constants $m_0>0$ such that for any  $x,x'\in\mathbb{R}^{d}$,
				\begin{align*}
					m_0|x-x'|^{2}+\langle x-x',b_1(t,x)-b_1(t,x')\rangle\leq \beta|x-x'|^{2},
				\end{align*}
				where $\langle \;, \;\rangle$ is the usual inner product in $\mathbb{R}^d$.
				\smallskip
				
				\item For $b_2(t,x,\omega)$, $\sigma(t,x,\omega)$, and $f(t,x,u,\omega)$ being continuous at $x\in\mathbb{R}^d$ and for any $\xi, ~\xi'\in L^{2p}(\Omega;\, \mathbb{R}^d)$ with some $p\geq2$, the following inequalities are satisfied:
				\begin{align*}
					&\mathbb{E}\|\sigma(t,\xi)\|^{2p}+\mathbb{E}|b_2(t,\xi)|^{2p}+\mathbb{E}\int_{U}|f(t,\xi,u)|^{2p}\nu(du)\leq H(t,\mathbb{E}|\xi|^{2p}),\\
					&\mathbb{E}\|\sigma(t,\xi)-\sigma(t,\xi')\|^{2}+\mathbb{E}|b_2(t,\xi)-b_2(t,\xi')|^{2}+\mathbb{E}\int_{U}|f(t,\xi,u)-f(t,\xi',u)|^{2}\nu(du)\\
					&\hspace{7.5cm}
					\leq G(t,\mathbb{E}|\xi-\xi'|^{2}),
				\end{align*}
				where the norm $\|\cdot\|$ is the usual norm induced by the inner product in $\mathbb{R}^d$, and $H$ satisfies (H2) and $G(t,x)$ is locally integrable in $t$ for any fixed $x$ and is continuous and nondecreasing in $x$ for any fixed $t$ and $G(t,0)=0$ for any $t$. Furthermore, 
				if $Z(t)$ is a nonnegative function  satisfying that for all $t\in[0,T]$ where $T>0$ and any constant $c>0$,
				\begin{align*}
					Z(t)\leq c\int_{0}^{t}G(s,Z(s))ds,
				\end{align*}
				one has $Z(t)\equiv0$ for all $t\in[0,T]$.
			\end{enumerate}
			Then there exists a unique solution $(X,\eta)$ of  Eqn.   (\ref{eqn:example1}).
		\end{theorem}
		
		\begin{remark}	
			\label{remark1}
			(1). When $\varphi(x)=
			\begin{cases}
				0,&\hspace{-0.4cm}x\in D,\\
				+\infty, & \hspace{-0.4cm}x\notin D,
			\end{cases}$ 
			where $D$ is a closed convex domain in $\mathbb{R}^d$, then $\partial\varphi$ has the following form:
			$ 
			\partial\varphi(x)=
			\begin{cases}
				\{0\},& \hspace{-0.4cm}x\in D,\\
				\Pi(x),   & \hspace{-0.4cm}x\in \partial D,\\
				\emptyset, & \hspace{-0.4cm}x\notin D,\end{cases}
			$
			where $\Pi(x)$ denotes the exterior normal cone to $D$ at $x$, and Eqn. (\ref{eqn:example1}) becomes a SDE reflected in $D$ with variable delays.  \smallskip
			
			(2). Let $g:\mathbb{R}\to(-\infty,+\infty]$ be a function which is  convex on $(0,+\infty)$ and $g\in \mathcal{C}^1(0,+\infty)$, such that $g=+\infty$ in $(-\infty,0]$ and $g(0+)=+\infty$. For $x=(x^1,x^2,\ldots,x^d)\in\mR^d$ where $d\geq1$, let
			$$
			\varphi(x)=\begin{cases}
				\sum_{1\leq i<j\leq d}g(x^j-x^i), \quad& ~~x^1<x^2<...<x^d,\\
				+\infty, & ~~\text{otherwise}.
			\end{cases}
			$$
			Then $\varphi$ is a lower semicontinuous convex function with domain $D(\partial\varphi)=\Big\{x\in\mR^d;\, x^1<x^2<\ldots<x^d\Big\}$. Eqn. (\ref{eqn:example1}) generalizes various interacting praticle systems with inter-particles repulsions. In particular, when $g(x^j-x^i)=-\lambda \ln(x^j-x^i)$ for $x^j>x^i$, then $\varphi$ is the special form  given in Eqn. \eqref{eqn:special_form} of Section \ref{sec:Modeling_strength_illustration}. 
			
		\end{remark}	
		\subsection{Infinite-dimensional example}
		\label{sec:examples2}
		Let $O$ be an open bounded domain in $\mathbb{R}^{d}$ with  a smooth boundary $\partial O$. 
		Consider the Sobolev space defined in the usual sense as 
		$$W^{k,2}(O)=L^2(O)\cap \big\{u:\,D^{\alpha}u\in  L^2(O),\,|\alpha|\leq k\big\},$$
		where $|\alpha|=\sum_{i=1}^{d}\alpha_i$ and the $\alpha$-th order partial derivative
		$D^{\alpha}=\frac{\partial^{|\alpha|}}{\partial x_1^{\alpha_1}\cdots \partial x_d^{\alpha_d}}.$
		The space $\mathbb{V}=W^{1,2}(O)$ is equipped with the norm 
		$$\| v\|_{\mathbb{V}}^2:=\int_O v(x)^2dx+\int_O |\nabla v|^2dx.$$
		The space $W_0^{1,2}(O)$ is defined as the closure of $C_0^{\infty}(O)$. It is known that $\Big(W_{0}^{1,2}(O), L^{2}(O), W^{-1,2}(O)\Big)$ forms a Gelfand triple. Consider the following SPDE:
		\begin{equation}\label{eqn:SPDE}
			\begin{split}
				\begin{cases}
					du(t,x)\in -\left[\sum_{i=1}^d\frac{\partial}{\partial x_{i}}h_{i}(x,\nabla_x u(t,x))+g(u(t,x))\right]dt+b(t,x,u_t(x))dt+\sigma(t,x,u_t(x))dW(t)\\
					\hspace*{1.7cm}
					+\int_{U}f(t,x,u_t(x),y)\widetilde{N}(dt,dy)-\partial\varphi(u(t,x))dt,\hspace{1cm}t\in[0,T],\\
					u(t,x)=\phi(t,x), \hspace{7.1cm}t\in[-h_0,0],
				\end{cases}
			\end{split}		
		\end{equation}
		where $h_{i}(x,v):\mathbb{R}^{d}\times\mathbb{R}^{d}\rightarrow\mathbb{R}$ for $i=1,\cdots, d$ is continuous with respect to $v$, $g(u):\mathbb{R}\rightarrow\mathbb{R}$ is continuous, and
		\begin{align*}
			&b(t,x,u_t(x)):=b^{(1)}\big(t,x,u(t,x)\big)+b^{(2)}\big(t,x,u(t-\delta_{1}(t),x)\big)+\int_{-\delta_{2}(t)}^0b^{(3)}\big(t,x,u(t+s,x)\big)ds,\\
			&\sigma(t,x,u_t(x)):=\sigma^{(1)}\big(t,x,u(t,x)\big)+\sigma^{(2)}\big(t,x,u(t-\delta_{1}(t),x)\big)+\int_{-\delta_{2}(t)}^0\sigma^{(3)}\big(t,x,u(t+s,x)\big)ds,\\
			&f(t,x,u_t(x),y):=f^{(1)}\big(t,x,u(t,x),y\big)+f^{(2)}\big(t,x,u(t-\delta_{1}(t),x),y\big)\\
			&\hspace{8.2cm}+\int_{-\delta_{2}(t)}^0f^{(3)}\big(t,x,u(t+s,x),y\big)ds.
		\end{align*}
		Here, $\delta_{i}$ for $i=1, 2$ are continuous functions from $[0,T]$ to $[0,h_0]$, $b^{(i)}$ and $f^{(i)}$ for $i=1,2,3$ are continuous functions from $[0,\infty)\times O\times\mathbb{R}$ to $\mathbb{R}$, and  $\sigma^{(i)}$ for $i=1, 2, 3$ are continuous functions from $[0,\infty)\times O\times\mathbb{R}$ to $\mathbb{R}\times l^2$. 
		\begin{remark}	
			\label{remark2}
			(1). When $\sum_{i=1}^d\frac{\partial}{\partial x_{i}}h_{i}(x,\nabla_x u(t,x))=\mathrm{div}\big(B(x)\nabla_xu(t,x)\big)+h(x)\cdot \nabla_xu(t,x)$ and $ \varphi, ~g\equiv0$, where $|h|\in L^p(O)$ for some $p>d$,   $B(x):=(c_{ij}(x))$ is matrix-valued, and there exists some constant $c_0\geq1$ such that
			$$
			\frac1{c_0}I_{d\times d}\leq B(x)\leq c_0I_{d\times d},
			$$
			and  $b^{(i)}, \sigma^{(i)}, f^{(i)}=0$ for $i=2, 3$, Eqn. \eqref{eqn:SPDE} is the SPDE discussed in \cite{zhen2011stochastic}. In particular, when $ b, f, \sigma^{(2)}, \sigma^{(3)}=0$ and $B=I_{d\times d}$ is the identity matrix, then the above 
			equation is the classical stochastic reaction-diffusion equation which has been studied extensively (see for example \cite{da2014stochastic}, \cite{liu2015stochastic}  and references therein). 
			\smallskip
			
			(2). When $f=0$ and  $b^{(i)}, \sigma^{(i)}=0$ for $i=2, 3$,
			Eqn. \eqref{eqn:SPDE} is the multivalued SPDE discussed in Section 6.2 of \cite{zhang2007skorohod}. 
		\end{remark}	
		
		\begin{theorem} 
			\label{thm:wellposed_example2}
			Suppose the following assumptions hold:
			\begin{enumerate}
				\item There exists some $m_0>0$ such that
				\begin{align*}
					\sum_{i=1}^{d}(v_{i}-v'_{i})(h_{i}(x,v)-h_{i}(x,v'))\geq  m_0|v-v'|^{2},\qquad \forall v, v', x\in\mathbb{R}^{d},
				\end{align*}
				Moreover, there exist  $\lambda, \beta\geq0$, $\mu\in L^{2}(O)$ such that
				\begin{align*}
					&(u_{1}-u_{2})(g(u_{1})-g(t, u_{2}))\geq-\beta|u_{1}-u_{2}|^{2},\qquad  \forall x\in\mathbb{R}^{d},\ \forall u_{1},u_{2}\in\mathbb{R},\\
					&	|g(v)|+|h_{i}(x,u)|\leq \mu(x)+\lambda(|v|+|u|).
				\end{align*}
				\item There exist $l'\in L^{2}([0,T]\times O)$ and $\rho, ~\gamma, ~\kappa$ as being continuous, nondecreasing and concave functions on $\mathbb{R}^+$ satisfying that $\rho(0)=\gamma(0)=\kappa(0)=0$ and
				$$\int_{0+}\frac{x}{\rho^2(x)+\gamma^2(x)+\kappa^2(x)}dx=+\infty,$$ such that for any $\widehat{v},\widehat{v}'\in L^{2}(O)$ and $~i=1,2,3$, 
				\begin{align*}
					&|b^{(i)}(t,x,\widehat{v})-b^{(i)}(t,x,\widehat{v}')|^2\leq l'(t,x)\gamma(\|\widehat{v}-\widehat{v}'\|_{L^2(O)}^2),\\
					&\|\sigma^{(i)}(t,x,\widehat{v})-\sigma^{(i)}(t,x,\widehat{v}')\|_{L_{2}^{0}}^{2}
					\leq l'(t,x)\rho(\|\widehat{v}-\widehat{v}'\|_{L^2(O)}^2), \\												&\|\sigma^{(i)}(t,x,0)\|^4_{L_{2}^{0}}+|b^{(i)}(t,x,0)|^4+\Big(\int_{U}|f^{(i)}(t,x,0,y)|^{2}\nu(dy)\Big)^2+\int_{U}|f^{(i)}(t,x,0,y)|^{4}\nu(dy)\\
					&\hspace{9cm}\leq l'^2(t,x),\\	
					&\int_{U}\Big|f^{(i)}(t,x,\widehat{v},y)-f^{(i)}(t,x,\widehat{v}',y)\Big|^2\nu(dy)
					\leq l'(t,x)\kappa(\|\widehat{v}-\widehat{v}'\|_{L^2(O)}^2).
				\end{align*}
				
				\item $|\varphi(x)|\leq C_{0}(1+|x|^{2})$ for any $ x\in\mathbb{R}$ and $\phi$ is continuous on $[-h,0]\times\mathbb{R}$.
			\end{enumerate}
			Then  Eqn.   (\ref{eqn:SPDE}) has a unique solution.
		\end{theorem}
		
		\begin{proof}
			Define an operator $A(t,\cdot)$ from $\mathbb{V}$ to $\mathbb{V}^{*}$ as:
			\begin{align*}
				A(t,\widehat{v})=-g(\widehat{v})-\sum_{i=1}^d\frac{\partial}{\partial x_{i}}h_{i}(x,\nabla_x \widehat{v}),\qquad \widehat{v}\in \mathbb{V}.
			\end{align*}
			Then $A$ satisfies (H1). Indeed, for any $\widehat{v}, ~\widehat{v}'\in\mathbb{V}$,
			\begin{align*}
				{}_{\mathbb{V}}\big\langle \widehat{v}-\widehat{v}',A(t,\widehat{v})-A(t,\widehat{v}')\big\rangle_{\mathbb{V}^{*}}
				=&-\int_{O}(\widehat{v}(x)-\widehat{v}'(x))\big(g(\widehat{v}(x))-g(\widehat{v}'(x))\big)dx\\
				&-\sum_{i=1}^{d}\int_{O}\big(\nabla_{i} \widehat{v}(x)-\nabla_{i} \widehat{v}'(x)\big)\Big(h_{i}(x,\nabla_x \widehat{v}(x))-h_{i}(x,\nabla_x \widehat{v}'(x))\Big)dx\\
				\leq&\beta\int_{O}|\widehat{v}(x)-\widehat{v}'(x)|^{2}dx-m_0\int_{O}|\nabla_x \widehat{v}(x)-\nabla_x \widehat{v}'(x)|^{2}dx.
			\end{align*}
			
			It is easy to see that with conditions on $b$, $\sigma$, and $f$, we have (H2) and (H3) hold, and (H4) holds by the third condition of this theorem.
			Thus, by Theorem \ref{wellposed}, there exists a unique solution of  Eqn.  (\ref{eqn:SPDE}).
		\end{proof}

\appendix

\section{Properties of $\varphi_\epsilon$}
\label{appendix:Properties}
\begin{proposition}
	\label{prop:varphi_epsilon}
	For $\epsilon>0$, define $\varphi_{\epsilon}, J_{\epsilon}: \mathbb{H} \rightarrow \mathbb{R}$ as
	$$
	\varphi_{\epsilon}(u):=\inf \left\{\frac{1}{2}|v-u|^{2}+\epsilon \varphi(v) ; \; v \in \mathbb{H}\right\}\quad\text{and}\quad
	J_{\epsilon}(u):=(I+\epsilon \partial\varphi)^{-1}(u).$$
	Then we have the following properties:
	\begin{enumerate}
		\item $\varphi_{\epsilon}$ is a convex $C^1$-function and
		$J_{\epsilon}(u)$ is Lipschitz continuous such that
		$$\|J_{\epsilon}(u)-J_{\epsilon}(u')\|_{\mathbb{H}}\leq \|u-u'\|_{\mathbb{H}}\quad\text{and}\quad \lim_{\epsilon\rightarrow 0}J_{\epsilon}(u)=\text{Proj}_{\overline{D(\partial\varphi)}}(u).$$
		\item For any $u,v\in \mathbb{H}$, $D\varphi_{\epsilon}(u)=u-J_{\epsilon}(u)$, and 
		\begin{align*}
			\langle D\varphi_{\epsilon}(u),v-u\rangle_{\mathbb{H}}\leq \varphi_{\epsilon}(v)-\varphi_{\epsilon}(u)\leq \epsilon[ \varphi(v)-\varphi(u) ].
		\end{align*}
		\item For any $u\in \mathbb{H}$, $\varphi_{\epsilon}(u)\geq \varphi_{\epsilon}(0)=0$, $J_{\epsilon}(0)=0$, $D\varphi_{\epsilon}(0)=0$, and
		\begin{align*}
			\varphi_{\epsilon}(u)=\frac{1}{2}|u-J_{\epsilon}(u)|^2+\epsilon \varphi(J_{\epsilon}(u)),
		\end{align*}
		\item For any $u,v\in \mathbb{H}$, $\frac{1}{\epsilon} D \varphi_{\epsilon}(u) \in \partial \varphi\left(J_{\epsilon}(u)\right)$ and 
		$$\left\langle\frac{1}{\epsilon} D \varphi_{\epsilon}(u),\; J_{\epsilon}(u)-v\right\rangle_{\mathbb{H}} \geqslant \varphi\left(J_{\epsilon}(u)\right)-\varphi(v).$$
		\item For any $u,v\in \mathbb{H}$ and $\delta>0$, 
		$$\left\langle\frac{1}{\delta} D \varphi_{\delta}(v)-\frac{1}{\epsilon} D \varphi_{\epsilon}(u),\; v-u\right\rangle_{\mathbb{H}} \geqslant-\left(\frac{1}{\delta}+\frac{1}{\epsilon}\right)\left\langle D \varphi_{\epsilon}(u), D \varphi_{\delta}(v)\right\rangle_{\mathbb{H}}.$$
		\item For any $u,v\in \mathbb{H}$,  we have $\langle D \varphi_{\epsilon}(u), u-J_{\epsilon} v \rangle_{\mathbb{H}} \geqslant-|v|^{2}$ and
		\begin{align*}
			&\frac{1}{2}\left\|D \varphi_{\epsilon}(u)\right\|_{\mathbb{H}}^{2} \leq \varphi_{\epsilon}(u) \leq\langle D \varphi_{\epsilon}(u), u\rangle_{\mathbb{H}} \leq\|u\|_{\mathbb{H}}^{2},\\
			&\langle D \varphi_{\epsilon}(u), u-v\rangle_{\mathbb{H}} \geqslant-\epsilon \varphi\left(J_{\epsilon} (v)\right)-\langle D \varphi_{\epsilon}(u), D \varphi_{\epsilon}(v)\rangle_{\mathbb{H}}.
		\end{align*}
	\end{enumerate}
\end{proposition}

\section{Proof of Theorem \ref{wellposed1}} 
\label{appendix:wellposed1}
\begin{proof}
	For $\epsilon>0$, for $\varphi_{\epsilon}: \mathbb{H} \rightarrow \mathbb{R}$ defined in Proposition \ref{prop:varphi_epsilon},
	by Theorem 1 in \cite{taniguchi2010existence}, for any $\varepsilon>0$, a  unique solution $X^{\epsilon}$ exists for the following equation:
	\begin{align}\label{see_omega}
		\begin{cases}
			dX^{\epsilon}(t)= 
			\sigma(t, \omega) d W(t)+b(t, \omega)d t+\int_{U} f(t, u, \omega)\widetilde{N}(ds,du)\\
			\hspace*{1.6cm}+A(t,X^{\epsilon}(t))dt-\frac{1}{\epsilon} D \varphi_{\epsilon}(X^{\epsilon}(t))dt, \hspace*{2.5cm}t\geq 0,\\
			X^{\epsilon}(t)=\phi(t), \hspace*{6.5cm}t\in[-h,0].
		\end{cases}
	\end{align}
	Since $0 \in \operatorname{Int}(D(\partial \varphi)), \varphi$ is locally bounded on $ \operatorname{Int}(D(\partial \varphi))$, and there exists $\gamma_{0}>0$ and $m_{0}>0$ such that
	$
	\varphi(\gamma_{0} h) \leq M_{0}$ for any $h \in \mathbb{H}$ and $\|h\|_{\mathbb{H}} \leq 1$. 
	By Proposition 1.2 in \cite{bensoussan1997stochastic}, for any $u \in \mathbb{H}$,
	\begin{align}
		\label{eqn:delta1.r}
		\gamma_{0}\left\|D \varphi_{\epsilon}(u)\right\|_{\mathbb{H}} \leq \epsilon M_{0}+\langle D \varphi_{\epsilon}(u), u\rangle_{\mathbb{H}} .
	\end{align}
	By the energy equality in Theorem 1 in \cite{taniguchi2010existence}, and (H1) in Assumption \ref{assumption}, 
	\begin{align*}
		\left\|X^{\epsilon}(t)\right\|_{\mathbb{H}}^{2}=&\|\phi(0)\|_{\mathbb{H}}^{2}+2 \int_{0}^{t}{}_{\mathbb{V}}\big\langle  X^{\epsilon}(s), A(s, X^{\epsilon}(s))\big\rangle_{\mathbb{V}^{*}}d s +2 \int_{0}^{t}\big\langle X^{\epsilon}(s), b(s,\omega)\big\rangle_{\mathbb{H}} d s\nonumber\\ 
		& +2 \int_{0}^{t}\big\langle X^{\epsilon}(s), \sigma(s,\omega)\big\rangle d W(s)-\frac{2}{\epsilon} \int_{0}^{t}\big\langle X^{\epsilon}(s),  D \varphi_{\epsilon}(X^{\epsilon}(s))\big\rangle_{\mathbb{H}}ds \nonumber\\
		& +2 \int_{0}^{t} \int_U\big\langle  X^{\epsilon}(s), f(s,u,\omega) \big\rangle_{\mathbb{H}} \widetilde{N}(ds,du)+\int_{0}^{t} \int_{U}\|f(s,u,\omega)\|_{\mathbb{H}}^{2} \widetilde{N}(ds,du)\nonumber\\
		& +\int_{0}^{t} \int_{U}\|f(s,u,\omega)\|_{\mathbb{H}}^{2} \nu(ds,du)+\int_{0}^{t}\|\sigma(s,\omega)\|_{L_{2}^0}^{2} d s\nonumber\\
		\leq & \|\phi(0)\|_{\mathbb{H}}^{2}+\int_{0}^{t}\|X^{\epsilon}(s) \|_{\mathbb{H}}^{2} d s-\frac{2}{\epsilon} \int_{0}^{t}\big\langle X^{\epsilon}(s), D \varphi_{\epsilon}(X^{\epsilon}(s))\big\rangle_{\mathbb{H}} d s\nonumber\\
		& +\int_{0}^{t}\left[\|b(s, \omega)\|_{\mathbb{H}}^{2}+\|\sigma(s, \omega)\|_{L_{2}^0}^{2}+\int_{U}\|f(s, u, \omega)\|_{\mathbb{H}}^{2} \nu(d u)\right] d s \nonumber\\
		& +\beta \int_{0}^{t}\left\|X^{\epsilon}(s)\right\|_{\mathbb{H}}^{2} d s+2 \int_{0}^{t}\left\|X^{\epsilon}(s)\right\|_{\mathbb{V}}\|A(s, 0)\|_{\mathbb{V}^{*}} d s -m_{0} \int_{0}^{t}\left\|X^{\epsilon}(s)\right\|_{\mathbb{V}}^{2} d s \nonumber\\
		& +2 \int_{0}^{t} \int_{U}\big\langle X^{\epsilon}(s), f(s, u, \omega)\big\rangle_{\mathbb{H}}\widetilde{N}(ds,du)+\int_{0}^{t} \int_{U}\|f(s, u, \omega)\|_{\mathbb{H}}^{2} \widetilde{N}(ds,du)\nonumber\\
		&+2 \int_{0}^{t}\big\langle X^{\epsilon}(s), \sigma(s, \omega) d W(s)\big\rangle_{\mathbb{H}}.
	\end{align*}
	By Proposition \ref{prop:varphi_epsilon}, we have
	$$-\big\langle X^{\epsilon}(s), D \varphi_{\epsilon}(X^{\epsilon}(s))\big\rangle_{\mathbb{H}} \leq \varphi_{\epsilon}(0)-\varphi_{\epsilon}\left(X^{\epsilon}(s)\right) \leq 0,$$
	and hence
	\begin{align}\label{varphie}
		&\hspace{-1cm}\frac{2}{\epsilon} \int_{0}^{t}\big\langle X^{\epsilon}(s), D \varphi_{\epsilon}(X^{\epsilon}(s))\big\rangle_{\mathbb{H}} d s+\left\|X^{\epsilon}(t)\right\|_{\mathbb{H}}^{2}\nonumber\\
		\leq & \|\phi(0)\|_{\mathbb{H}}^{2}+(1+\beta)\int_{0}^{t}\|X^{\epsilon}(s) \|_{\mathbb{H}}^{2} d s+C_{m_0}\int_{0}^{t}\|A(s, 0)\|^2_{\mathbb{V}^{*}} d s -\frac{m_{0}}{2} \int_{0}^{t}\left\|X^{\epsilon}(s)\right\|_{\mathbb{V}}^{2} d s\nonumber\\
		& +\int_{0}^{t}\left[\|b(s, \omega)\|_{\mathbb{H}}^{2}+\|\sigma(s, \omega)\|_{L_{2}^0}^{2}+\int_{U}\|f(s, u, \omega)\|_{\mathbb{H}}^{2} \nu(d u)\right] d s\nonumber \\
		& +2 \int_{0}^{t} \int_{U}\left\langle X^{\epsilon}(s), f(s, u, \omega)\right\rangle_{\mathbb{H}}\widetilde{N}(ds,du)+\int_{0}^{t} \int_{U}\|f(s, u, \omega)\|_{\mathbb{H}}^{2} \widetilde{N}(ds,du)\nonumber\\
		&+2 \int_{0}^{t}\left\langle X^{\epsilon}(s), \sigma(s, \omega) d W(s)\right\rangle_{\mathbb{H}}.
	\end{align}
	By the BDG inequality,
	\begin{align*}
		\mathbb{E} \sup _{t \leq T}\left|2\int_{0}^{t}\big\langle X^{\epsilon}(s), \sigma(s, \omega) d W(s)\big\rangle_{\mathbb{H}}\right|^{2} \leq &C_T\mathbb{E} \int_{0}^{T}\|X^{\epsilon}(s)\|_{\mathbb{H}}^{2}\|\sigma(s, \omega)\|_{L_{2}^0}^{2} d s\\
		\leq & \frac{1}{4} \mathbb{E} \sup _{t \leq T}\left\|X^{\epsilon}(t)\right\|_{\mathbb{H}}^{4} +C_{T}  \mathbb{E} \left(\int_{0}^{T}\|\sigma(s, \omega)\|_{L_{2}^0}^{2} d s\right)^{2}.
	\end{align*}
	By Lemma \ref{lem2}, we have
	\begin{align*}
		& \mathbb{E} \sup _{t \leq T}\left(\int_{0}^{t} \int_{U}\|f(s, u, \omega)\|_{\mathbb{H}}^{2} \widetilde{N}(d s, d u)\right)^{2} \leq \mathbb{E} \int_{0}^{T} \int_{U}\|f(s, u, \omega)\|_{\mathbb{H}}^{4} v(d u) d s
	\end{align*}
	and
	\begin{align*}
		& \hspace{-1cm}\mathbb{E} \sup _{t\leq T}\left(\int_{0}^{t} \int_{U} 2\big\langle X^{\epsilon}(s), f(s, u, \omega)\big\rangle_{\mathbb{H}} \widetilde{N}(d s, d u)\right)^{2} \\
		& \leq C_{T} \mathbb{E} \int_{0}^{T} \int_{U}\left\|X^{\epsilon}(s)\right\|_{\mathbb{H}}^{2} \cdot\|f(s, u, \omega)\|_{\mathbb{H}}^{2} v(d u) d s \\
		& \leq \frac{1}{4} \mathbb{E} \sup _{t \leq T}\left\|X^{\epsilon}(t)\right\|_{\mathbb{H}}^{4}+C_{T} \mathbb{E} \int_{0}^{T}\left(\int_{U}\|f(s, u, \omega)\|^{2} v(d u)\right)^{2} d s.
	\end{align*}
	Hence, summing up the above results and using Gr\"onwall's lemma, we obtain
	\begin{align}\label{momes}
		\sup_{\epsilon} \mathbb{E} \sup_{t\leq T}\left\|X^{\epsilon}(t)\right\|_{\mathbb{H}}^{4} 
		\leq &C_{T} \mathbb{E}\|\phi(0)\|_{\mathbb{H}}^{4}+C_{T} \mathbb{E}\int_{0}^{T}\left[\|b(s)\|_{\mathbb{H}}^{2}+\|\sigma(s)\|_{L_{2}^0}^{2}+\int_{U}\|f(s, u)\|_{\mathbb{H}}^{2}\nu(d u)\right]^{2}ds\nonumber \\
		& +C_{T} \mathbb{E} \int_{0}^{T} \int_{U}\|f(s, u)\|_{\mathbb{H}}^{4} \nu(d u) d s < +\infty,	 
	\end{align}
	and
	\begin{align}\label{momes1}
		& \hspace{-0.5cm}	\sup_{\epsilon} \frac{m_{0}}{2} \mathbb{E}\left(\int_{0}^{T}\left\|X^{\epsilon}(s)\right\|_{\mathbb{V}}^{2} d s\right)^{2}\nonumber\\ 
		& \leq  C_{T}\mathbb{E}\|\phi(0)\|_{\mathbb{H}}^{4}+C_{T}\mathbb{E} \int_{0}^{T}\Big[\|b(s)\|_{\mathbb{H}}^{2}+\left\|\sigma(s)\right\|_{L_{2}^{0}}^{2} +\int_{u}\|f(s, u)\|_{\mathbb{H}}^{2} v(d u)\Big]^{2}ds \nonumber\\
		& \quad+C_{T} \mathbb{E} \int_{0}^{T} \int_{U}\|f(s, u)\|_{\mathbb{H}}^{4} \nu(d u) d s  < +\infty.
	\end{align}
	It follows from Eqn. \eqref{varphie} that
	\begin{align}
		\label{eqn:delta2.r}
		\sup_{\epsilon}\mathbb{E} \Big(\int_{0}^{t}\frac{2}{\epsilon}\big\langle X^{\epsilon}(s), D \varphi_{\epsilon}(X^{\epsilon}(s))\big\rangle_{\mathbb{H}} d s\Big)^2< +\infty.
	\end{align}
	Applying the energy equality in Theorem 1 of \cite{taniguchi2010existence} to $\left\|X^{\epsilon}(t)-X^{\delta}(t)\right\|_{\mathbb{H}}^{2}$, and by Proposition \ref{prop:varphi_epsilon}, we obtain
	\begin{align*}
		\left\|X^{\epsilon}(t)-X^{\delta}(t)\right\|_{\mathbb{H}}^{2}=&2 \int_{0}^{t}{}_{\mathbb{V}}\Big\langle X^{\epsilon}(s)-X^{\delta}(s), A\left(s, X^{\epsilon}(s)\right)-A(s, X^{\delta}(s))\Big\rangle_{\mathbb{V}^{*}} d s \\
		& -2\int_{0}^{t}\Big\langle X^{\epsilon}(s)-X^{\delta}(s), \frac{1}{\epsilon} D \varphi_{\epsilon}\left(X^{\epsilon}(s)\right)-\frac{1}{\delta} D \varphi_{\delta}(X^{\delta}(s))\Big\rangle_{\mathbb{H}} d s  \\
		\leq	&  2 \beta \int_{0}^{t}\left\|X^{\epsilon}(s)-X^{\delta}(s)\right\|_{\mathbb{H}}^{2} d s-2 m_{0} \int_{0}^{t}\left\|X^{\epsilon}(s)-X^{\delta}(s)\right\|_{\mathbb{V}}^{2} d s \\
		& +2\left(\frac{1}{\epsilon}+\frac{1}{\delta}\right) \int_{0}^{t}\left\langle D \varphi_{\epsilon}\left(X^{\epsilon}(s)\right), D \varphi_{\delta}(X^{\delta}(s))\right\rangle_{\mathbb{H}} d s .
	\end{align*}
	Furthermore, 
	\begin{align*}
		&\hspace{-0.5cm}\mathbb{E} \sup _{t\leq T}\left\|X^{\epsilon}(t)-X^{\delta}(t)\right\|_{\mathbb{H}}^{2}\\ &\leq  \beta \mathbb{E} \int_{0}^{T}\left\|X^{\epsilon}(s)-X^{\delta}(s)\right\|_{\mathbb{H}}^{2} d s+2 \mathbb{E} \sup_{t\leq T}\left\|D \varphi_{\epsilon}\left(X^{\epsilon}(s)\right)\right\|_{\mathbb{H}}\left\|\frac{1}{\delta} D \varphi_{\delta}(X^{\delta}(s))\right\|_{\mathbb{H}} d s\\
		&\quad+2 \mathbb{E} \sup _{t\leq T}\int_{0}^{t}\left\|D \varphi_{\delta} (X^{\delta}(s) )\right\|_{\mathbb{H}}\left\|\frac{1}{\epsilon} D \varphi_{\epsilon}\left(X^{\epsilon}(s)\right)\right\|_{\mathbb{H}} d s\\
		& \leq 2 \beta \mathbb{E} \int_{0}^{T}\left\|X^{\epsilon}(s)-X^{\delta}(s)\right\|_{\mathbb{H}}^{2} d s\\
		&\quad+2\left\{\mathbb{E} \sup _{t \leq T}\left\|D \varphi_{\epsilon}\left(X^{\epsilon}(t)\right)\right\|_{\mathbb{H}}^{2} \cdot \mathbb{E}\left(\int_{0}^{T} \frac{1}{\delta}\left\|D \varphi_{\delta}(X^{\delta}(s))\right\|_{\mathbb{H}} d s\right)^{2}\right\}^{1 / 2}\\
		&\quad+2\left\{\mathbb{E} \sup_{t \leq T}\left\|D \varphi_{\delta}(X^{\delta}(t))\right\|_{\mathbb{H}}^{2} \cdot \mathbb{E}\left(\int_{0}^{T} \frac{1}{\epsilon}\left\|D \varphi_{E}\left(X^{\epsilon}(s)\right)\right\|_{\mathbb{H}} d s\right)^{2}\right\}^{1/2}.
	\end{align*}
	Note that, by Proposition \ref{prop:varphi_epsilon}, we have
	\begin{align*}
		\mathbb{E}\sup_{t \leq T}\left\|D \varphi_{\epsilon}(X^{\epsilon}(t))\right\|_{\mathbb{H}}^{2}\leq \mathbb{E}\sup_{t \leq T} \varphi_{\epsilon}(X^{\epsilon}(t))\leq \epsilon \mathbb{E}\sup_{t \leq T}\big[\varphi(X^{\epsilon}(t))-\varphi(0)\big] 
		\leq C\epsilon.
	\end{align*}
	Then, by  Eqn. \eqref{eqn:delta1.r} and Eqn. \eqref{eqn:delta2.r}, we have
	\begin{align}\label{varDvarphi}
		\sup_{\delta} \mathbb{E}\left(\int_{0}^{T}\frac{1}{\delta}\| D \varphi_{\delta}(X^{\delta}(s))\|_{\mathbb{H}}ds\right)^2
		& \leq 2\left(\frac{M_{0} T}{\gamma_{0}}\right)^{2}+C \sup_{\delta} \mathbb{E}\left(\int_{0}^{T}\left\langle\frac{1}{\delta} D \varphi_{\delta}(X^{\delta}(s)), X^{\delta}(s)\right\rangle d s \right)^{2}\nonumber \\
		& <+\infty.
	\end{align}
	
	Then we have by Gr\"onwall's lemma that as $\epsilon, \delta \rightarrow 0$,
	$$
	\mathbb{E}\sup_{t \leq T} \| X^{\epsilon}(t)-X^{\delta}(t) \|_{\mathbb{H}}^2 \leqslant C\left(\epsilon^{1/2}+\delta^{1/2}\right)\rightarrow 0 ,
	$$
	and 
	$$
	\mathbb{E}\int_0^T \| X^{\epsilon}(t)-X^{\delta}(t) \|_{\mathbb{V}}^2dt\rightarrow 0.
	$$
	Hence, $\{X^{\epsilon}\}$ is a Cauchy sequence which converges to some   
	$X$ in  $$L^2( [0, T]\times \Omega ;\mathbb{V})\cap L^2(\Omega; D([0,T];\mathbb{H})).$$
	
	Set $$\eta^{\epsilon}(t):=\frac{1}{\epsilon} \int_0^t D \varphi_{\epsilon}\left(X^{\epsilon}(s)\right) d s.$$ Then by  Eqn. \eqref{varDvarphi}, we have
	\begin{equation}\label{etamom}
		\sup_{\epsilon}\mathbb{E}\big(|\eta^{\epsilon}|_{0}^{T}\big)^2<+\infty.
	\end{equation}
	Recall that $A$ is bounded and thus by equations \eqref{momes}-\eqref{momes1} and \eqref{etamom}, there exists a
	space $\Omega_0\subset\Omega$ with $P(\Omega_0)=1$ and a subsequence $\left\{\epsilon_k\right\}$, such that for all $\omega \in \Omega_0$, $\liminf_{\epsilon_k} |\eta^{\epsilon_k}(\omega)|_{0}^{T}<+\infty$, and as $\epsilon_k\to0$, 
	\begin{equation}\label{ABconV}\begin{aligned}
			\sup_{-h\leq t\leq T} \| X^{\epsilon_k}(t,\omega)-X(t,\omega)\|_{\mathbb{H}}^2\rightarrow 0,\qquad\int_{0}^{T}\| X^{\epsilon_k}(t,\omega)-X(t,\omega)\|_{\mathbb{V}}^2dt \rightarrow 0,\\
			\text{and}\quad 	A\left(\cdot, X^{\epsilon_k}(\cdot, \omega)\right) \longrightarrow B(\cdot, \omega) \quad\text { weakly in } L^2\left([0, T] ; \mathbb{V}^*\right).
	\end{aligned}\end{equation}
	Note that
	for any $v \in L^2([0, T] ; \mathbb{V})$, $\delta>0$, $\omega \in \Omega_0$, and $t\in[0,T]$,
	\begin{align} 
		\label{eqn:Appendixeqn1}
		&\hspace{-0.3cm}\int_{0}^{t}{}_{\mathbb{V}}\Big\langle v(s), B(s,\omega)-A(s,X(s,\omega)-\delta v(s))\Big\rangle_{\mathbb{V}^{*}}ds\nonumber\\
		=&\int_{0}^{t}{}_{\mathbb{V}}\Big\langle v(s), B(s,\omega)-A(s,X^{\epsilon_k}(s,\omega))\Big\rangle_{\mathbb{V}^{*}}ds\\
		&+\frac1{\delta}\int_{0}^{t}{}_{\mathbb{V}}\Big\langle X^{\epsilon_k}(s,\omega)-X(s,\omega)+\delta v(s),\, A\left(s, X^{\epsilon_k}(s,\omega)\right)-A(s, X(s,\omega)-\delta v(s))\Big\rangle_{\mathbb{V}^{*}}ds\nonumber\\
		&-\frac1{\delta}\int_{0}^{t}{}_{\mathbb{V}}\Big\langle X^{\epsilon_k}(s,\omega)-X(s,\omega), \,A\left(s, X^{\epsilon_k}(s,\omega)\right)-A(s, X(s,\omega)-\delta v(s))\Big\rangle_{\mathbb{V}^{*}}ds.\nonumber
	\end{align}
	By (H1) in Assumption \ref{assumption}, we have 
	\begin{align*}   
		&\hspace{-0.3cm}\int_{0}^{t}{}_{\mathbb{V}}\Big\langle X^{\epsilon_k}(s,\omega)-X(s,\omega)+\delta v(s), \,A\left(s, X^{\epsilon_k}(s,\omega)\right)-A(s, X(s,\omega)-\delta v(s))\Big\rangle_{\mathbb{V}^{*}}ds\\
		\leq& \beta \int_{0}^{t}\| X^{\epsilon_k}(s,\omega)-X(s,\omega)+\delta v(s) \|_{\mathbb{H}}^2ds,\\
		\mbox{and}\\
		&-\frac1{\delta}\int_{0}^{t}{}_{\mathbb{V}}\Big\langle X^{\epsilon_k}(s,\omega)-X(s,\omega), \,A\left(s, X^{\epsilon_k}(s,\omega)\right)-A(s, X(s,\omega)-\delta v(s))\Big\rangle_{\mathbb{V}^{*}}ds\\
		\leq &\frac C{\delta}\int_{0}^{t}\| X^{\epsilon_k}(s,\omega)-X(s,\omega) \|_{\mathbb{V}}ds.
	\end{align*}
	Sending $\epsilon_k \rightarrow 0$ of equation \eqref{eqn:Appendixeqn1} and applying \eqref{ABconV}, we obtain
	$$\int_{0}^{t}{}_{\mathbb{V}}\Big\langle v(s), B(s,\omega)-A(s,X(s,\omega)-\delta v(s))\Big\rangle_{\mathbb{V}^{*}}ds\leq \beta \delta\int_{0}^{t}\|v(s) \|_{\mathbb{H}}^2ds.$$
	Now sending to $\delta\to0$, we obtain by the hemicontinuity of $A$ that 
	$$\int_{0}^{t}{}_{\mathbb{V}}\Big\langle v(s), B(s,\omega)-A(s,X(s,\omega))\Big\rangle_{\mathbb{V}^{*}}ds\leq 0.
	$$
	The arbitrariness of ${v}$ leads to $B(s, \omega)=A(s, X(s,\omega))$ for almost all $s \in[0,t]$.
	
	Set for any $\omega\in\Omega_0$,
	\begin{align*}
		\eta{(t,\omega)}=&\phi{(0)}+\int_0^t b(s, \omega) d s+\int_0^t \sigma(s, \omega) d W(s)+\int_0^t \int_U f(s, u, \omega) \widetilde{N}(d s, d u)\\
		&+\int_0^t A(s, X(s,\omega)) d s-X(t,\omega).
	\end{align*}
	Then for any $t \in[0, T]$,
	$$\eta^{\varepsilon_k}(t, \omega) \longrightarrow \eta{(t, \omega)} \quad\text{weakly in}\; \mathbb{V}^*.$$	
	Now, it suffices to show that $|\eta(\cdot,\omega)|_{0}^{T}<\infty$ and for any $\alpha\in \mathcal{D}([0,T];\mathbb{H})$ and any $0\leq s\leq t\leq T$,
	\begin{align}\label{81}
		\int_s^t\big\langle X(r,\omega)-\alpha(r),d\eta(r,\omega)\big\rangle_{\mathbb{H}}\geq \int_s^t\big[\varphi(X(r,\omega))-\varphi(\alpha(r))\big]dr.
	\end{align}
	To this end, let $0=t_{1}<t_{2}<\cdots<t_{m}=T$ be a partition of $[0,T]$ and set $\alpha^{m}(t):=\sum_{j=1}^{m}a_{j}\mathbbm{1}_{[t_{j},t_{j+1})}(t)$ where $a_{j}\in\mathbb{V}$  such that
	$$\lim_{m\to\infty}\sup_{t\leq T}\|\alpha^m(t)-\alpha(t)\|_{\mathbb{H}}=0.$$
	On one hand, note that
	\begin{align*}
		\sum_{j=1}^{m}\|\eta(t_{j+1},\omega)-\eta(t_{j},\omega)\|_{\mathbb{H}}&=\sum_{j=1}^{m}\sup_{v\in\mathbb{V},\|v\|_{\mathbb{H}}\leq 1}\big|\big\langle v,\eta(t_{j+1},\omega_{0})-\eta(t_{j},\omega)\big\rangle_{\mathbb{H}}\big|\\
		&\leq\sum_{j=1}^{m}\lim_{\epsilon_k\rightarrow0}\big\|\eta^{\epsilon_k}(t_{j+1},\omega)-\eta^{\epsilon_k}(t_{j},\omega)\big\|_{\mathbb{H}}\\
		&\leq\sup_{\epsilon_k}|\eta^{\epsilon_k}(\cdot,\omega)|_{0}^{T}<\infty.
	\end{align*}
	Taking supremum over partitions of $[0,T]$ gives $$|\eta(\cdot,\omega)|_{0}^{T}<\infty.$$
	On the other hand, by the lower semicontinuity of $\varphi$, for any $0\leq s\leq t\leq T$,
	\begin{align*}
		&\hspace{-1cm}\int_s^t\big[\varphi(X(r,\omega))-\varphi(\alpha(r))\big]dr\\
		\leq& \liminf_{\epsilon_k\rightarrow0}\int_s^t\big[\varphi(X^{\epsilon_k}(r,\omega))-\varphi(\alpha(r))\big]dr\\
		\leq&\liminf_{\epsilon_k\rightarrow0}\int_s^t\big\langle X^{\epsilon_k}(r,\omega),d\eta^{\epsilon_k}(r,\omega)\big\rangle_{\mathbb{H}}
		-\liminf_{\epsilon_k\rightarrow0}\int_s^t\langle\alpha(r),d\eta^{\epsilon_k}(r,\omega)\rangle_{\mathbb{H}}.
	\end{align*}
	Moreover,
	\begin{align}
		&\hspace{-1cm}\Bigg|\int_{s}^{t}\big\langle X^{\epsilon_k}(r,\omega),d\eta^{\epsilon_k}(r,\omega)\big\rangle_{\mathbb{H}}-\int_{s}^{t}\big\langle X(r,\omega_{0}),d\eta^{\epsilon_k}(r,\omega)\big\rangle_{\mathbb{H}}\Bigg|\nonumber\\
		\leq&\sup_{s\leq r\leq t}\big\|X^{\epsilon_k}(r,\omega)-X(r,\omega)\big\|_{\mathbb{H}}\sup_{\epsilon_k}|\eta^{\epsilon_k}(\cdot,\omega)|_{0}^{T}\nonumber\\
		\leq& C\sup_{s\leq r\leq t}\big\|X^{\epsilon_k}(r,\omega)-X(r,\omega)\big\|_{\mathbb{H}}\nonumber\\
		&\rightarrow0,\hspace{5cm} \text{as}\quad \epsilon_k\rightarrow0. \label{08}
	\end{align}
	And \begin{align}\label{081}
		\int_s^t\langle X(r,\omega),d\eta^{\epsilon_k}(r,\omega)\big\rangle_{\mathbb{H}}\to \int_s^t\langle X(r,\omega),d\eta(r,\omega)\big\rangle_{\mathbb{H}},\quad \text{as}\quad \epsilon_k\rightarrow0.
	\end{align}
	For any $\delta'>0$, choose $k_0$ sufficiently large such that for all $k>k_0$ and all $j=1,\ldots,m$,
	\begin{align*}
		\Big|{}_{\mathbb{V}}\langle a_j, \eta^{\epsilon_k}(t_{j+1},\omega)-\eta^{\epsilon_k}(t_{j},\omega)\rangle_{\mathbb{V}^{*}}-{}_{\mathbb{V}}\langle a_j,\eta(t_{j+1},\omega)-\eta(t_{j},\omega)\rangle_{\mathbb{V}^{*}}\Big|\leq\frac{\delta'}{2m}.
	\end{align*}
	Denote $$L:=\sup_{k}|\eta^{\epsilon_k}(\cdot,\omega)|_{0}^{T}\vee|\eta(\cdot,\omega)|_{0}^{T}.$$ 
	Take $m$ sufficiently large such that $$\sup_{t\leq T}\|\alpha^m(t)-\alpha(t)\|_{\mathbb{H}}<\frac{\delta'}{4L}.$$ It then follows that
	\begin{align}\label{011}
		&\hspace{-1cm}\Bigg|\int_s^t \Big\langle\alpha(r),d\eta^{\epsilon_k}(r,\omega)-d\eta(r,\omega)\Big\rangle_{\mathbb{H}}\Bigg|\nonumber\\
		&\leq \Bigg|\int_s^t \Big\langle\alpha(r)-\alpha^m(r),d\eta^{\epsilon_k}(r,\omega)-d\eta(r,\omega)\Big\rangle_{\mathbb{H}}\Bigg|\nonumber\\
		&\quad+
		\sum_{j=1}^m\Bigg|{}_{\mathbb{V}}\Big\langle a_j, \eta^{\epsilon_k}(t_{j+1},\omega)-\eta^{\epsilon_k}(t_{j},\omega)-\eta(t_{j+1},\omega)+\eta(t_{j},\omega)\Big\rangle_{\mathbb{V}^{*}}\Bigg|\nonumber\\
		&\leq 2L\frac{\delta'}{4L}+m\frac{\delta'}{2m}\nonumber\\
		&=\delta'.
	\end{align}
	Combining  Eqn.   (\ref{08}) and  Eqn.   (\ref{011}), we obtain  Eqn.   (\ref{81}).
	
	At last, we aim to establish the uniqueness. Suppose $(X'(t),\eta'(t))$ is another solution. Similar to Proposition \ref{prop1}, we have 
	$$\mathbb{E}\sup_{-h\leq t\leq T}(\|X(t)\|_{\mathbb{H}}^{2}\vee|X'(t)\|_{\mathbb{H}}^{2})<\infty.$$ 
	By the energy equality again, we obtain
	\begin{align*}
		&\|X(t)-X'(t)\|_{\mathbb{H}}^{2}\\
		&=2\int_{0}^{t}\Bigg[{}_{\mathbb{V}}\Big\langle X(s)-X'(s), A(s,X(s))-A(s,X'(s))\Big\rangle_{\mathbb{V}^{*}}-\Big\langle X(s)-X'(s),d\eta(s)-d\eta'(s)\Big\rangle_{\mathbb{H}}\Bigg]ds\\
		&\leq \beta\int_{0}^{t}\|X(s)-X'(s)\|^2_{\mathbb{H}}ds-m_0\int_{0}^{t}\|X(s)-X'(s)\|^2_{\mathbb{V}}ds.
	\end{align*}
	We have by Gr\"onwall's lemma that
	\begin{align*}
		\mathbb{E}\sup_{-h\leq s\leq T}\|X(s)-X'(s)\|_{\mathbb{H}}^{2}=0\quad\text{and}\quad
		\mathbb{E}\int_{-h}^{T}\|X(s)-X'(s)\|_{\mathbb{V}}^{2}ds=0,
	\end{align*}
	which complete the proof of pathwise uniqueness.
\end{proof}

\section{Proof of Proposition \ref{prop1}} 
\label{appendix:prop1}
\begin{proof}
	It is easy to see that
	\begin{equation}\begin{split}
			\label{phi}
			\sup_n \mathbb{E}\sup_{t\in[-h,0]}\|X^{n}(t)\|_{\mathbb{H}}^4=\mathbb{E}\sup_{t\in[-h,0]}\|\phi(t)\|_{\mathbb{H}}^4<\infty,\\
			\sup_n \mathbb{E}\int_{-h}^0 \|X^{n}(t)\|_{\mathbb{V}}^2dt=\mathbb{E}\int_{-h}^0 \|\phi(t)\|_{\mathbb{V}}^2dt<\infty.
		\end{split}
	\end{equation}
	Applying the energy equality in Theorem \ref{wellposed1} to $e^{-\beta t}\|X^{n+1}(t)\|_{\mathbb{H}}^{2}$, we have that for $0\leq t\leq T$,
	\begin{align*}
		&e^{-\beta t}\|X^{n+1}(t)\|_{\mathbb{H}}^{2}\\
		&=\|\phi({0})\|_{\mathbb{H}}^{2}+2\int_{0}^{t}e^{-\beta s}\big\langle X^{n+1}(s),\sigma(s,X^{n}_{s})dW(s)\big\rangle_{\mathbb{H}}+2\int_{0}^{t}e^{-\beta s}\big\langle X^{n+1}(s),b(s,X^{n}_{s})\big\rangle_{\mathbb{H}}ds\\
		&\quad+2\int_{0}^{t}e^{-\beta s}{}_{\mathbb{V}}\big\langle X^{n+1}(s), A(s,X^{n+1}(s))\big\rangle_{\mathbb{V}^{*}}ds-2\int_{0}^{t}e^{-\beta s}\big\langle X^{n+1}(s),d\eta^{n+1}(s)\big\rangle_{\mathbb{H}}\\
		&\quad+2\int_{0}^{t}e^{-\beta s}\int_{U}\big\langle X^{n+1}(s),f(s,X^{n}_{s},u)\big\rangle_{\mathbb{H}}\widetilde{N}(ds,du)+\int_{0}^{t}e^{-\beta s}\Big[\|\sigma(s,X^{n}_{s})\|_{L_{2}^{0}}^{2}-\beta\|X^{n+1}(s)\|_{\mathbb{H}}^{2}\Big]ds\\
		&\quad+\int_{0}^{t}e^{-\beta s}\int_{U}\|f(s,X^{n}_{s},u)\|_{\mathbb{H}}^{2}\nu(du)ds+\int_{0}^{t}\int_{U}e^{-\beta s}\|f(s,X^{n}_{s},u)\|_{\mathbb{H}}^{2}\widetilde{N}(ds,du).
	\end{align*}
	From condition (H1), we have that
	\begin{align*}
		2{}_{\mathbb{V}}\big\langle X^{n+1}(t),A(t,X^{n+1}(t))\big\rangle_{\mathbb{V}^{*}}
		&\leq {}_{\mathbb{V}}\big\langle X^{n+1}(t)-0, A(t,X^{n+1}(t))-A(t,0)\big\rangle_{\mathbb{V}^{*}}+\|X^{n+1}(t)\|_{\mathbb{V}}\|A(t,0)\|_{\mathbb{V}^{*}}\\
		&\leq\beta\|X^{n+1}(t)\|_{\mathbb{H}}^{2}-\frac{m_{0}}{2}\|X^{n+1}(t)\|_{\mathbb{V}}^{2}+C.
	\end{align*}	
	Note that
	\begin{align*}
		&\hspace{-1cm}\mathbb{E}\sup_{t\in[0,T]}\left|2\int_{0}^{t}e^{-\beta s}\big\langle X^{n+1}(s),b(s,X^{n}_{s})\big\rangle_{\mathbb{H}}ds\right|^2\\
		&\leq\frac16\mathbb{E}\sup_{t\in[0,T]}e^{-2\beta t}\|X^{n+1}(t)\|^4_{\mathbb{H}}
		+C_T\mathbb{E}\int_0^Te^{-2\beta s}\|b(s,X^{n}_{s})\|^4_{\mathbb{H}}ds.
	\end{align*}	
	Applying the BDG inequality, we obtain
	\begin{align*}
		&\hspace{-1cm}\mathbb{E}\sup_{t\in[0,T]}\left|2\int_{0}^{t}e^{-\beta s}\big\langle X^{n+1}(s),\;\sigma(s,X^{n}_{s})dW(s)\big\rangle_{\mathbb{H}}ds\right|^2\\
		&\leq\frac16\mathbb{E}\sup_{t\in[0,T]}e^{-2\beta t}\|X^n(t)\|^4_{\mathbb{H}}
		+C_T\mathbb{E}\int_0^Te^{-2\beta s}\|\sigma(s,X^{n}_{s})\|^4_{L_2^0}ds.
	\end{align*}	
	By Lemma \ref{lem2}, we have
	\begin{align*}
		&\hspace{-1cm}\mathbb{E}\sup_{0\leq s\leq T}\Bigg|2\int_{0}^{s}\int_{U}e^{-\beta r}\big\langle X^{n+1}(r),f(r,X^{n}_{r},u)\big\rangle_{\mathbb{H}}\widetilde{N}(dr,du)\Bigg|^2\\
		&\leq\mathbb{E}\int_{0}^{T}\int_{U}e^{-2\beta r}\|X^{n+1}(r)\|^2_{\mathbb{H}}
		\|f(r,X^{n}_{r},u)\|^2_{\mathbb{H}}\nu(du)dr\\
		&\leq \frac{1}{6}\mathbb{E}\sup_{0\leq t
			\leq T}e^{-2\beta t}\|X^{n+1}(t)\|_{\mathbb{H}}^{4}+C\mathbb{E}\int_{0}^{T}\left(\int_{U}e^{-\beta s}\|f(s,X^{n}_{s},u)\|_{\mathbb{H}}^{2}\nu(du)\right)^2ds,
	\end{align*}
	and
	\begin{align*}
		&\mathbb{E}\sup_{0\leq t\leq T}\Bigg|\int_{0}^{t}\int_{U}e^{-\beta r}\|f(r,X^{n}_{r},u)\|^2_{\mathbb{H}}\widetilde{N}(dr,du)\Bigg|^2
		\leq	C\mathbb{E}\int_{0}^{T}\int_{U}e^{-2\beta s}\|f(s,X^{n}_{s},u)\|_{\mathbb{H}}^{4}\nu(du)ds.
	\end{align*}
	Summing up the above estimates gives 
	\begin{equation}\label{6}
		\begin{split}	
			&\mathbb{E}\sup_{0\leq s\leq T}e^{-2\beta s}\|X^{n+1}(s)\|_{\mathbb{H}}^{4}+\frac{m_{0}^2}{4}\mathbb{E}\left(\int_{0}^{T}e^{-\beta s}
			\|X^{n+1}(s)\|_{\mathbb{V}}^{2}ds\right)^2\\
			&\leq C(1+\mathbb{E} \|\phi({0})\|_{\mathbb{H}}^{4})+\frac{1}{2}\mathbb{E}\sup_{0\leq s\leq T}e^{-2\beta s}\|X^{n+1}(s)\|_{\mathbb{H}}^{4}+C_T\int_{0}^{T}H\Big(s,\mathbb{E}\sup_{r\in[-h,s]}\|X^{n}(r)\|_{\mathbb{H}}^{4}\Big)ds,
		\end{split}
	\end{equation}
	which implies that
	\begin{equation}\label{xn1}
		\begin{split}
			\mathbb{E}\sup_{0\leq s\leq T}\|X^{n+1}(s)\|_{\mathbb{H}}^{4}
			\leq C_T(1+\|\phi\|_{[-h,0]}^4)+C_{T}\int_{0}^{T}H\Big(s,\mathbb{E}\sup_{r\in[-h,s]]}\|X^{n}(r)\|_{\mathbb{H}}^{4}\Big)ds.
		\end{split}    	
	\end{equation}
	By (H2), there exists a continuous nondecreasing function $z(t)$ satisfying
	$$
	z(t)=z(0)+C_{T}\int_0^tH(s,z(s))ds,
	$$
	where $z(0):=C_{T}(1+\|\phi\|_{[-h,0]}^{4})$. We now claim that
	\begin{align*}
		\mathbb{E}\sup_{-h\leq s\leq T}\|X^{n+1}(s)\|_{\mathbb{H}}^{4}\leq z(t).	
	\end{align*}
	Indeed, for $n=0$, we have $$\mathbb{E}\|X^{0}\|_{[-h,T]}^{4}=\|\phi\|_{[-h,0]}^{4}\leq z(0)\leq z(t),
	$$ and then
	\begin{align*}
		\mathbb{E}\sup_{-h\leq s\leq t}\|X^{1}(s)\|_{\mathbb{H}}^{4}\leq&C_{T}(1+\mathbb{E}\|\phi\|_{[-h,0]}^{4})+C_{T}\int_{0}^{t}H\Big(s,\mathbb{E}\sup_{-h\leq r\leq s}\|X^0(r)\|_{\mathbb{H}}^{4}\Big)ds\\
		\leq&z(0)+C_{T}\int_{0}^{t}H(s,z(s))ds\\
		=&z(t).
	\end{align*}
	Now supposing 
	\begin{align*}
		\mathbb{E}\sup_{-h\leq s\leq t}\|X^{n}(s)\|_{\mathbb{H}}^{4}\leq z(t),
	\end{align*}
	we have from Eqn. (\ref{xn1}) that
	\begin{align*}
		\mathbb{E}\sup_{-h\leq s\leq t}\|X^{n+1}(s)\|_{\mathbb{H}}^{4}
		&\leq C_{T}(1+\mathbb{E}\|\phi\|_{[-h,0]}^{4})+C_{T}\int_{0}^{t}H\Big(s,\mathbb{E}\sup_{-h\leq r\leq s}\|X^{n}(r)\|_{\mathbb{H}}^{4}\Big)ds\\
		&\leq z(0)+C_{T}\int_{0}^{t}H(s,z(s))ds\\
		&=z(t).
	\end{align*}
	The continuity of $z$ and  Eqn.   (\ref{phi}) yield
	\begin{align}\label{7}
		\sup_n\mathbb{E}\sup_{-h\leq s\leq T}\|X^{n}(s)\|_{\mathbb{H}}^{4}<\infty.
	\end{align}
	Plugging  Eqn.   (\ref{7}) into Eqn.   (\ref{6}) gives
	\begin{align*}
		\sup_n\mathbb{E}\left(\int_{0}^{T}\|X^{n}(s)\|_{\mathbb{V}}^{2}ds\right)^2<\infty.
	\end{align*}
	Applying the energy equality to $\|X^{n+1}(t)-a_0\|_{\mathbb{H}}^{2}$ and by Lemma \ref{lem-1}, we have
	\begin{align*}
		&\mathbb{E}\sup_{0\leq s\leq T}\|X^{n+1}(s)-a_0\|_{\mathbb{H}}^{4}+m_{1}^2\mathbb{E}\big(|\eta^{n+1}|_{0}^{T}\big)^2\\
		& \leq 1+\mathbb{E}\|\phi({0})-a_0\|_{\mathbb{H}}^{4}+\frac{1}{2}\mathbb{E}\sup_{0\leq s\leq T}\|X^{n+1}(s)-a_0\|_{\mathbb{H}}^{4}+\int_{0}^{T}(\beta+k_1)^2\mathbb{E}\|X^{n+1}(s)-a_0\|_{\mathbb{H}}^{4}ds\\
		&\quad+C\int_{0}^{T}H\Big(s,\mathbb{E}\sup_{-a(s)\leq r\leq 0}\|X^{n}(r+s)\|_{\mathbb{H}}^{4}\Big)ds+k_1^2m_1^2T.
	\end{align*}
	Then we have by  Eqn.   (\ref{7}) that
	\begin{align*}
		\sup_{n}\mathbb{E}\big(|\eta^{n+1}|_{0}^{T}\big)^2<\infty,
	\end{align*}
	which completes the proof.
\end{proof}

\section{Proof of Theorem \ref{wellposed}} 
\label{appendix:wellposed}
\begin{proof}
	Set $$\tau^{n}_{R}:=\inf\Big\{t\geq0;\;|X^{n}(t)\|_{\mathbb{H}}>R\Big\},\quad \tau^{m,n}_{R}:=\tau^{m}_{R}\wedge \tau^{n}_{R}.$$ We have by applying the energy equality that, for $0\leq t\leq T\wedge\tau^{m,n}_{R}$,
	\begin{align}
		\label{eqn:Appendix_d0}
		&e^{-\beta t}\|X^{m}(t)-X^{n}(t)\|_{\mathbb{H}}^{2}\\
		&=2\int_{0}^{t}e^{-\beta s}\Bigg[\Big\langle X^{m}(s)-X^{n}(s),(\sigma(s,X^{m-1}_{s})-\sigma(s,X^{n-1}_{s}))dW(s)\Big\rangle_{\mathbb{H}}\nonumber\\
		&\hspace{2.9cm}+\Big\langle X^{m}(s)-X^{n}(s),b(s,X^{m-1}_{s})-b(s,X^{n-1}_{s})\Big\rangle_{\mathbb{H}}ds\nonumber\\
		&\hspace{2.9cm}+{}_{\mathbb{V}}\Big\langle X^{m}(s)-X^{n}(s),A(s,X^{m}(s))-A(s,X^{n}(s))\Big\rangle_{\mathbb{V}^{*}}ds\nonumber\\
		&\hspace{2.9cm}+\int_{U}\Big\langle X^{m}(s)-X^{n}(s),f(s,X^{m-1}_{s},u)-f(s,X^{n-1}_{s},u) \Big\rangle_{\mathbb{H}}\widetilde{N}(ds,du)\nonumber\\
		&\hspace{2.9cm}-\Big\langle X^{m}(s)-X^{n}(s),d\eta^{m}(s)-d\eta^{n}(s)\Big\rangle_{\mathbb{H}}\Bigg]\nonumber\\
		&\quad+\int_{0}^{t}e^{-\beta s}\Bigg[\Big\|\sigma(s,X^{m-1}_{s})-\sigma(s,X^{n-1}_{s})\Big\|^{2}_{L_{2}^{0}}ds-\beta\|X^{m}(s)-X^{n}(s)\|_{\mathbb{H}}^{2}ds\nonumber\\
		&\hspace{2.9cm}+\int_{U}\Big\|f(s,X^{m-1}_{s},u)-f(s,X^{n-1}_{s},u)\Big\|_{\mathbb{H}}^{2}\nu(du)ds\nonumber\\
		&\hspace{2.9cm}+\int_{U}\Big\|f(s,X^{m-1}_{s},u)-f(s,X^{n-1}_{s},u)\Big\|_{\mathbb{H}}^{2}\widetilde{N}(ds,du)\Bigg]\nonumber\\
		&\leq\frac14\sup_{s\in[0,t]}e^{-\beta s}\|X^{m}(s)-X^{n}(s)\|_{\mathbb{H}}^2\nonumber\\
		&\quad+2\int_{0}^{t}e^{-\beta s}\Bigg[
		2\|b(s,X^{m-1}_{s})-b(s,X^{n-1}_{s})\|^2_{\mathbb{H}}ds\nonumber\\
		&\hspace{3cm}+\Big\langle X^{m}(s)-X^{n}(s),(\sigma(s,X^{m-1}_{s})-\sigma(s,X^{n-1}_{s}))dW(s)\Big\rangle_{\mathbb{H}}\nonumber\\
		&\hspace{3cm}+\frac{\beta}2\|X^{m}(s)-X^{n}(s)\|_{\mathbb{H}}^2-m_0\|X^{m}(s)-X^{n}(s)\|_{\mathbb{V}}^2\nonumber\\
		&\hspace{3cm}+\int_{U}\Big\langle X^{m}(s)-X^{n}(s),f(s,X^{m-1}_{s},u)-f(s,X^{n-1}_{s},u) \Big\rangle_{\mathbb{H}}\widetilde{N}(ds,du)\Bigg]\nonumber\\
		&\quad+\int_{0}^{t}e^{-\beta s}\Bigg[\Big\|\sigma(s,X^{m-1}_{s})-\sigma(s,X^{n-1}_{s})\Big\|^{2}_{L_{2}^{0}}ds-\beta\|X^{m}(s)-X^{n}(s)\|_{\mathbb{H}}^{2}ds\nonumber\\
		&\hspace{3cm}+\int_{U}\Big\|f(s,X^{m-1}_{s},u)-f(s,X^{n-1}_{s},u)\Big\|_{\mathbb{H}}^{2}\nu(du)ds\nonumber\\
		&\hspace{3cm}+\int_{U}\Big\|f(s,X^{m-1}_{s},u)-f(s,X^{n-1}_{s},u)\Big\|_{\mathbb{H}}^{2}\widetilde{N}(ds,du)\Bigg]\nonumber\\
		&\leq \frac16\sup_{s\in[0,t]}e^{-\beta s}\|X^{m}(s)-X^{n}(s)\|_{\mathbb{H}}^2
		-2m_0\int_{0}^{t}e^{-\beta s}\|X^{m}(s)-X^{n}(s)\|_{\mathbb{V}}^2ds\nonumber\\
		&\quad+2\int_{0}^{t}e^{-\beta s}\Bigg[
		2\|b(s,X^{m-1}_{s})-b(s,X^{n-1}_{s})\|^2_{\mathbb{H}}
		+\Big\|\sigma(s,X^{m-1}_{s})-\sigma(s,X^{n-1}_{s})\Big\|^{2}_{L_{2}^{0}}\nonumber\\
		&\hspace{3cm}+\int_{U}\Big\|f(s,X^{m-1}_{s},u)-f(s,X^{n-1}_{s},u)\Big\|_{\mathbb{H}}^{2}\nu(du)\Bigg]ds\nonumber\\
		&\quad+2\int_{0}^{t}e^{-\beta s}\Bigg[
		\Big\langle X^{m}(s)-X^{n}(s),(\sigma(s,X^{m-1}_{s})-\sigma(s,X^{n-1}_{s}))dW(s)\Big\rangle_{\mathbb{H}}\nonumber\\
		&\hspace{3cm}+\int_{U}\Big\langle X^{m}(s)-X^{n}(s),f(s,X^{m-1}_{s},u)-f(s,X^{n-1}_{s},u) \Big\rangle_{\mathbb{H}}\widetilde{N}(ds,du)\nonumber\\
		&\hspace{3cm}+\int_{U}\Big\|f(s,X^{m-1}_{s},u)-f(s,X^{n-1}_{s},u)\Big\|_{\mathbb{H}}^{2}\widetilde{N}(ds,du)\Bigg]  \nonumber        \end{align}
	Applying BDG's inequality and Lemma \ref{lem2} we have 
	\begin{align*}
		&\mathbb{E}\sup_{t\in[0,T]}\left|2\int_{0}^{t}e^{-\beta s}\Big\langle X^{m}(s)-X^{n}(s),(\sigma(s,X^{m-1}_{s})-\sigma(s,X^{n-1}_{s}))dW(s)\right|\\
		&\leq \frac16\mathbb{E}\sup_{s\in[0,T]}e^{-\beta s}\|X^{m}(s)-X^{n}(s)\|_{\mathbb{H}}^2
		+C\mathbb{E}\int_{0}^{T}e^{-\beta s}
		\Big\|\sigma(s,X^{m-1}_{s})-\sigma(s,X^{n-1}_{s})\Big\|^{2}_{L_{2}^{0}}ds.
	\end{align*}
	and
	\begin{align*} 
		&\mathbb{E}\sup_{t\in[0,T]}\left|2\int_{0}^{t}e^{-\beta s}\int_{U}\Big\langle X^{m}(s)-X^{n}(s),f(s,X^{m-1}_{s},u)-f(s,X^{n-1}_{s},u) \Big\rangle_{\mathbb{H}}\widetilde{N}(ds,du)\right|\\
		&\leq \frac16\mathbb{E}\sup_{s\in[0,T]}e^{-\beta s}\|X^{m}(s)-X^{n}(s)\|_{\mathbb{H}}^2
		+C\mathbb{E}\int_{0}^{T}e^{-\beta s}\Big(\int_{U}\Big\|f(s,X^{m-1}_{s},u)-f(s,X^{n-1}_{s},u)\Big\|_{\mathbb{H}}\nu(du)\Big)^{2}ds\\
		&\quad+C\mathbb{E}\int_{0}^{T}e^{-\beta s}\Big\|f(s,X^{m-1}_{s},u)-f(s,X^{n-1}_{s},u)\Big\|_{\mathbb{H}}^2\nu(du)ds.
	\end{align*}
	Taking expectations on both sides, and applying (H1) and (H3), we obtain
	\begin{align}
		\label{eqn:Appendix_d}
		&\mathbb{E}\sup_{0\leq s\leq T}\Big\|X^{m}(s\wedge\tau^{m,n}_{R})-X^{n}(s\wedge\tau^{m,n}_{R})\Big\|_{\mathbb{H}}^{2}+m_{0} \mathbb{E}\int_{0}^{T}\Big\|X^{m}(s\wedge\tau^{m,n}_{R})-X^{n}(s\wedge\tau^{m,n}_{R})\Big\|_{\mathbb{V}}^{2}ds\\
		&\leq e^{\beta T}\mathbb{E}\sup_{0\leq s\leq T}e^{-\beta s}\Big\|X^{m}(s\wedge\tau^{m,n}_{R})-X^{n}(s\wedge\tau^{m,n}_{R})\Big\|_{\mathbb{H}}^{2}\nonumber\\
		&\hspace{0.9cm}+m_{0}e^{\beta T}\mathbb{E}\int_{0}^{T}e^{-\beta s}\Big\|X^{m}(s\wedge\tau^{m,n}_{R})-X^{n}(s\wedge\tau^{m,n}_{R})\Big\|_{\mathbb{V}}^{2}ds\nonumber\\
		&\leq C_T\mathbb{E}\int_{0}^{T\wedge\tau^{m,n}_{R}}\Bigg[
		2\Big\|b(s,X^{m-1}_{s})-b(s,X^{n-1}_{s})\Big\|^2_{\mathbb{H}}
		+\Big\|\sigma(s,X^{m-1}_{s})-\sigma(s,X^{n-1}_{s})\Big\|^{2}_{L_{2}^{0}}\nonumber\\
		&\hspace{0.3cm}+\int_{U}\Big\|f(s,X^{m-1}_{s},u)-f(s,X^{n-1}_{s},u)\Big\|_{\mathbb{H}}^{2}\nu(du)
		+\Big(\int_{U}\Big\|f(s,X^{m-1}_{s},u)-f(s,X^{n-1}_{s},u)\Big\|_{\mathbb{H}}\nu(du)\Big)^{2}\Bigg]ds\nonumber\\
		&\leq C_{T,m_{0}}\int_{0}^{T}G_{R}\Big(s,\mathbb{E}\sup_{0\leq r\leq s}\Big\|X^{m-1}(r\wedge\tau^{m,n}_{R})-X^{n-1}(r\wedge\tau^{m,n}_{R})\Big\|_{\mathbb{H}}^{2}\Big)ds.\nonumber
	\end{align}
	With (H3) we have
	\begin{align*}
		\limsup_{m,n}\mathbb{E}\sup_{0\leq s\leq T}\Big\|X^{m}(s\wedge\tau^{m,n}_{R})-X^{n}(s\wedge\tau^{m,n}_{R})\Big\|_{\mathbb{H}}^{2}=0,
	\end{align*}
	and furthermore,
	\begin{align*}
		\limsup_{m,n\rightarrow\infty}\mathbb{E}\int_{0}^{T}\Big\|X^{m}(s\wedge\tau^{m,n}_{R})-X^{n}(s\wedge\tau^{m,n}_{R})\Big\|_{\mathbb{V}}^{2}ds=0.
	\end{align*}
	By Proposition \ref{prop1} we have 
	\begin{align*}
		&\hspace{-0.5cm}\limsup_n\mathbb{E}\Bigg(\sup_{t\in[-h,T]}\|X^{m}(t)-X^{n}(t)\|^2_{\mathbb{H}}\Bigg)\\
		\leq& 	\limsup_{m,n}\mathbb{E}\Bigg(\sup_{t\in[-h,T]}\Big\|X^{m}(t\wedge\tau^{m,n}_{R})-X^{n}(t\wedge\tau^{m,n}_{R})\Big\|^2_{\mathbb{H}}\Bigg)
		+	2\sup_n\mathbb{E}\Bigg(\sup_{t\in[-h,T]}\|X^{n}(t)\|^2_{\mathbb{H}}\mathbbm{1}_{\{T<\tau^{n}_{R}\}}\Bigg)\\
		\leq& \frac2{R^2}\sup_n\mathbb{E}\sup_{t\in[-h,T]}\|X^{n}(t)\|_{\mathbb{H}}^4\rightarrow  0,\qquad \mathrm{by ~letting}  ~~R\to\infty,\end{align*}
	and moreover,
	\begin{align*}
		\mathbb{E}\int_0^T\|X^{m}(t)-X^{n}(t)\|_{\mathbb{V}}^2dt\to0, \qquad \mathrm{as} ~m, n\to\infty.
	\end{align*}
	Therefore, by Propositions \ref{prop1} and Fatou's lemma, there exists a subsequence of $X^n$, denoted as $X^{n_l}$, such that as $l\to\infty$,
	\begin{align*}
		\sup_{0\le t\le T}\|X^{n_l}(t)-X(t)\|_{\mathbb{H}}\rightarrow0\quad\text{and}\quad
		\int_{0}^{T}\|X^{n_l}(s)-X(s)\|_{\mathbb{V}}^{2}ds\rightarrow0,~~~a.s..
	\end{align*}
	Set $X(t)=\phi(t)$ for $t\in[-h,0]$. Then $X\in \mathcal{D}([-h,T];\mathbb{H})$ and moreover,
	\begin{align*}
		\sup_{t\in[-h,T]}\|X(t)\|_{\mathbb{H}}^2<\infty\quad\text{and}\quad\int_0^T\|X(t)\|_{\mathbb{V}}^2dt<\infty, ~a.s..
	\end{align*}
	Therefore, one chooses $\Omega'\subset\Omega$ such that $P(\Omega')=1$ and for all $\omega_{0}\in\Omega'$,
	$$\liminf_{l\rightarrow\infty}|\eta^{n_{l}}(\cdot,\omega_{0})|_{0}^{T}=:c_0<\infty,$$
	\begin{align}\label{star}
		\lim_{l\rightarrow \infty}\sup_{-h\le t\le T}\|X^{n_{l}}(t,\omega_{0})-X(t,\omega_{0})\|_{\mathbb{H}}=0\quad\text{and}\quad\lim_{l\rightarrow\infty}\int_{0}^{T}\|X^{n_{l}}(s,\omega_{0})-X(s,\omega_{0})\|_{\mathbb{V}}^{2}ds=0.
	\end{align}
	The continuity of $\sigma$, $b$ and $f$ then implies that for all $\omega_{0}\in\Omega'$,
	\begin{align*}
		&\lim\limits_{l\rightarrow\infty}\sup_{0\leq t\leq T}\Bigg\|\int_{0}^{t}\sigma(s,X^{n_{l}}_{s}(\omega_{0}))dW(s)-\int_{0}^{t}\sigma(s,X_{s}(\omega_{0}))dW(s)\Bigg\|_{\mathbb{H}}=0,\\
		&\lim\limits_{l\rightarrow\infty}\Bigg\|\int_{0}^{T}b(s,X^{n_{l}}_{s}(\omega_{0}))ds-\int_{0}^{t}b(s,X_{s}(\omega_{0}))ds\Bigg\|_{\mathbb{H}}= 0,\\
		&\lim\limits_{l\rightarrow\infty}\sup_{0\leq t\leq T}\Bigg\|\int_{0}^{t}\int_{U}f(s,X^{n_{l}}_{s}(\omega_{0}),u)\widetilde{N}(ds,du)-\int_{0}^{t}\int_{U}f(s,X_{s}(\omega_{0}),u)\widetilde{N}(ds,du)\Bigg\|_{\mathbb{H}}=0.
	\end{align*}
	By the boundedness of $A$ in (H2),  for all $\omega_{0}\in\Omega'$, there exists a further subsequence $n_{l_{i}}$ of $n_l$ and $B\in L^{2}([0,T];\mathbb{V}^{*})$ such that
	\begin{align}\label{09}
		&A(\cdot,X^{n_{l_{i}}}(\cdot,\omega_{0}))\rightharpoonup B(\cdot,\omega_{0}),\qquad  \text{weakly in}\  L^{2}([0,T];\mathbb{V}^{*}).
	\end{align}
	By \eqref{star} we may choose a subsequence $\eta^{n_{l_{i}}}$ of $\eta^{n_{l}}$ such that $$
	\lim_{i\to\infty}|\eta^{n_{l_{i}}}(\cdot,\omega_0)|_0^T=c_0,$$ and thus 
	\begin{align}\label{090}
		\sup_{i}\sup_{0\leq t\leq T}\|\eta^{n_{l_{i}}}(t,\omega_{0})\|_{\mathbb{H}}\leq\sup_{i}|\eta^{n_{l_{i}}}(\cdot,\omega_{0})|_{0}^{T}<\infty.
	\end{align}	
	Note that, by (H1), for any $v\in L^2([0,T];\mathbb{V})$ and any $\delta>0$,
	\begin{align*}
		&2\int_0^T{}_{\mathbb{V}}\Big\langle v, A(s,X(s,\omega_{0})+\delta v)-B(s,\omega_{0})\Big\rangle_{\mathbb{V}^{*}}ds\\
		&=\frac2{\delta}\int_0^T{}_{\mathbb{V}}\Big\langle \delta v+X(s,\omega_{0})-X^{n_{l_i}}(s,\omega_{0}), A(s,X(s,\omega_{0})+\delta v)-A(s,X^{n_{l_i}}(s,\omega_{0}))\Big\rangle_{\mathbb{V}^{*}}ds\\
		&\hspace{3mm}-\frac2{\delta}\int_0^T{}_{\mathbb{V}}\Big\langle X(s,\omega_{0})-X^{n_{l_i}}(s,\omega_{0}), A(s,X(s,\omega_{0})+\delta v)-A(s,X^{n_{l_i}}(s,\omega_{0}))\Big\rangle_{\mathbb{V}^{*}}ds\\
		&\hspace{3mm}+2\int_0^T{}_{\mathbb{V}}\Big\langle v,A(s,X^{n_{l_i}}(s,\omega_{0}))-B(s,\omega_{0})\Big\rangle_{\mathbb{V}^{*}}ds\\
		&\leq\frac{2\beta}\delta\int_0^T\|X(s,\omega_{0})+\delta v-X^{n_{l_i}}(s,\omega_{0}))\|_{\mathbb{H}}^{2}ds+
		\frac{C}\delta\left(\int_0^T\|X(s,\omega_{0})-X^{n_{l_i}}(s,\omega_{0}))\|_{\mathbb{V}}^2ds\right)^{1/2}\\
		&\hspace{3mm}+2\int_0^T{}_{\mathbb{V}}\Big\langle v, A(s,X^{n_{l_i}}(s,\omega_{0}))-B(s,\omega_{0})\Big\rangle_{\mathbb{V}^{*}}ds.
	\end{align*}
	Plugging the convergence achieved earlier that $X^{n_{l_{i}}}(t,\omega_{0})$ converges to $X(t,\omega_{0})$ into equations \eqref{star} and \eqref{09} and sending $n_{l_{i}}\to\infty$, yield
	\begin{align*}
		2\int_0^T{}_{\mathbb{V}}\Big\langle v, A(s,X(s,\omega_{0})+\delta v)-B(s,\omega_{0})\Big\rangle_{\mathbb{V}^{*}}ds
		\leq {2\beta}\delta T\|v\|_{\mathbb{H}}^2.
	\end{align*}
	Sending $\delta\to0$ and noting that $A$ is hemicontinuous, we obtain for any $v\in\mathbb{V}$,
	$$
	\int_0^t {}_{\mathbb{V}}\Big\langle v, A(s,X(s,\omega_{0}))-B(s,\omega_{0})\Big\rangle_{\mathbb{V}^{*}}ds\leq 0.
	$$
	Since $v$ is arbitrary, it then follows that 
	$B(\cdot,\omega_{0})=A(\cdot,X(\cdot,\omega_{0}))$ in $L^{2}([0,T];\mathbb{V}^{*})$, 
	and moreover, by letting 
	\begin{align*}
		\eta(t,\omega_{0}):=&-X(t,\omega_{0})+\phi({0})+\int_{0}^{t}\sigma(s,X_{s}(\omega_{0}))dW(s)+\int_{0}^{t}A(s,X(s,\omega_{0}))ds\\
		&+\int_{0}^{t}b(s,X_{s}(\omega_{0}))ds+\int_{0}^{t}\int_{U}f(s,X_{s}(\omega_{0}),u)\widetilde{N}(ds,du),
	\end{align*}
	we have for all $\omega_{0}\in\Omega'$,
	\begin{align*}
		\lim_{i\rightarrow\infty}{}_{\mathbb{V}}\langle v,\eta^{n_{l_{i}}}(t,\omega_{0})\rangle_{\mathbb{V}^{*}}={}_{\mathbb{V}}\langle v,\eta(t,\omega_{0})\rangle_{\mathbb{V}^{*}},\qquad \forall v\in\mathbb{V}.
	\end{align*}	
	
	To show that $(X,\eta)$ is a solution of Eqn \eqref{see}, it remains to prove $(X(t,\omega_{0}),\eta(t,\omega_{0}))\in \Gr(\partial\varphi)$. It suffices to show that $|\eta(\cdot,\omega_{0})|_{0}^{T}<\infty$ and for any $\alpha\in \mathcal{D}([0,T];\mathbb{H})$ and any $0\leq s\leq t\leq T$,
	\begin{align}\label{810}
		\int_s^t\big\langle X(r,\omega_{0})-\alpha(r),d\eta(r,\omega_{0})\big\rangle_{\mathbb{H}}\geq \int_s^t\big[\varphi(X(r,\omega_{0}))-\varphi(\alpha(r))\big]dr.
	\end{align}
	Let $0=t_{1}<t_{2}<\cdots<t_{m}=T$ be a partition of $[0,T]$ and set $$\alpha^{m}(t):=\sum_{j=1}^{m}a_{j}\mathbbm{1}_{[t_{j},t_{j+1})}(t), \qquad a_{j}\in\mathbb{V},$$  so that
	$$\lim_{m\to\infty}\sup_{t\leq T}\|\alpha^m(t)-\alpha(t)\|_{\mathbb{H}}=0.$$
	On one hand, note that
	\begin{align*}
		\sum_{j=1}^{m}\Big\|\eta(t_{j+1},\omega_{0})-\eta(t_{j},\omega_{0})\Big\|_{\mathbb{H}}&=\sum_{j=1}^{m}\sup_{v\in\mathbb{V},\|v\|_{\mathbb{H}}\leq 1}\big|\big\langle v,\eta(t_{j+1},\omega_{0})-\eta(t_{j},\omega_{0})\big\rangle_{\mathbb{H}}\big|\\
		&\leq\sum_{j=1}^{m}\lim_{i\rightarrow\infty}\big\|\eta^{n_{l_{i}}}(t_{j+1},\omega_{0})-\eta^{n_{l_{i}}}(t_{j},\omega_{0})\big\|_{\mathbb{H}}\\
		&\leq\sup_{i}|\eta^{n_{l_{i}}}(\cdot,\omega_{0})|_{0}^{T}<\infty.
	\end{align*}
	Taking supremum over partitions of $[0,T]$ gives $$|\eta(\cdot,\omega_{0})|_{0}^{T}<\infty.$$
	On the other hand, by the lower semicontinuity of $\varphi$, for any $0\leq s\leq t\leq T$,
	\begin{align*}
		\int_s^t\big[\varphi(X(r,\omega_{0}))-\varphi(\alpha(r))\big]dr\leq& \liminf_{i\rightarrow\infty}\int_s^t\big[\varphi(X^{n_{l_{i}}}(r,\omega_{0}))-\varphi(\alpha(r))\big]dr\\
		\leq&\liminf_{i\rightarrow\infty}\int_s^t\big\langle X^{n_{l_{i}}}(r,\omega_{0}),d\eta^{n_{l_{i}}}(r,\omega_{0})\big\rangle_{\mathbb{H}}\\
		&-\liminf_{i\rightarrow\infty}\int_s^t\langle\alpha(r),d\eta^{n_{l_{i}}}(r,\omega_{0})\rangle_{\mathbb{H}}.
	\end{align*}
	Moreover, by equations \eqref{star} and \eqref{090}, 
	\begin{align}
		\label{080}
		&\hspace{-1cm}\Bigg|\int_{s}^{t}\big\langle X^{n_{l_{i}}}(r,\omega_{0}),d\eta^{n_{l_{i}}}(r,\omega_{0})\big\rangle_{\mathbb{H}}-\int_{s}^{t}\big\langle X(r,\omega_{0}),d\eta^{n_{l_{i}}}(r,\omega_{0})\big\rangle_{\mathbb{H}}\Bigg|\nonumber\\
		&\leq\sup_{s\leq r\leq t}\big\|X^{n_{l_{i}}}(r,\omega_{0})-X(r,\omega_{0})\big\|_{\mathbb{H}}\sup_{i}|\eta^{n_{l_{i}}}(\cdot,\omega_{0})|_{0}^{T}\nonumber\\
		&\leq C\sup_{s\leq r\leq t}\big\|X^{n_{l_{i}}}(r,\omega_{0})-X(r,\omega_{0})\big\|_{\mathbb{H}}\nonumber\\
		&\rightarrow0,\hspace{5cm} \text{as}\quad i\rightarrow\infty. 
	\end{align}
	Choose $N$ sufficiently large such that for all $i>N$ and all $j=1,\ldots,m$,
	\begin{align*}
		\Bigg|{}_{\mathbb{V}}\Big\langle a_j, \eta^{n_{l_{i}}}(t_{j+1},\omega_{0})-\eta^{n_{l_{i}}}(t_{j},\omega_{0})\Big\rangle_{\mathbb{V}^{*}}-{}_{\mathbb{V}}\Big\langle a_j,\eta(t_{j+1},\omega_{0})-\eta(t_{j},\omega_{0})\Big\rangle_{\mathbb{V}^{*}}\Bigg|\leq\frac{\epsilon}{2m}.
	\end{align*}
	Denote $$L:=\sup_i|\eta^{n_{l_{i}}}(\cdot,\omega_{0})|_{0}^{T}\vee|\eta(\cdot,\omega_{0})|_{0}^{T}.$$ 
	Take $m$ sufficiently large such that $$\sup_{t\leq T}\|\alpha^m(t)-\alpha(t)\|_{\mathbb{H}}<\frac{\epsilon}{4L}.$$ It then follows
	\begin{align}\label{0110}
		&\hspace{-1cm}\Bigg|\int_s^t\Big\langle\alpha(r),\;d\eta^{n_{l_{i}}}(r,\omega_{0})-d\eta(r,\omega_{0})\Big\rangle_{\mathbb{H}}\Bigg|\nonumber\\
		&\leq \Bigg|\int_s^t\Big\langle\alpha(r)-\alpha^m(r),\;d\eta^{n_{l_{i}}}(r,\omega_{0})-d\eta(r,\omega_{0}\Big\rangle_{\mathbb{H}}\Bigg|\nonumber\\
		&\quad+
		\sum_{j=1}^m\Bigg|{}_{\mathbb{V}}\Big\langle a_j,\; \eta^{n_{l_{i}}}(t_{j+1},\omega_{0})-\eta^{n_{l_{i}}}(t_{j},\omega_{0})-\eta(t_{j+1},\omega_{0})+\eta(t_{j},\omega_{0})\Big\rangle_{\mathbb{V}^{*}}\Bigg|\nonumber\\
		&\leq 2L\frac{\epsilon}{4L}+m\frac{\epsilon}{2m}\nonumber\\
		&=\epsilon.
	\end{align}
	Combining  Eqn.   (\ref{080}) and  Eqn.   (\ref{0110}), we obtain  Eqn.   (\ref{810}) and hence $(X,\eta)$ is a solution. Moreover, by Proposition \ref{prop1}, it follows that 
	\begin{align}\label{4thmom}
		\left(\mathbb{E}\sup_{-h\leq t\leq T}\|X(t)\|_{\mathbb{H}}^{4}+\mathbb{E}\left(\int_{0}^{T}\|X(t)\|_{\mathbb{V}}^{2}dt\right)^2+\mathbb{E}\left(|\eta|_{0}^{T}\right)^2\right)<\infty.
	\end{align}

	In the rest of the proof, we aim to establish the uniqueness. Suppose $(X'(t),\eta'(t))$ is another solution. Similar to Proposition \ref{prop1}, we have 
	$$\mathbb{E}\sup_{-h\leq t\leq T}(\|X(t)\|_{\mathbb{H}}^{2}\vee|X'(t)\|_{\mathbb{H}}^{2})<\infty.$$ 
	By the energy equality, we obtain
	\begin{align*}
		e^{-\beta t}\|X(t)-X'(t)\|_{\mathbb{H}}^{2}
		&=2\int_{0}^{t}e^{-\beta s}\Bigg[\Big\langle X(s)-X'(s),(\sigma(s,X_{s})-\sigma(s,X'_{s}))dW(s)\Big\rangle_{\mathbb{H}}\\
		&\hspace{3cm}+\Big\langle X(s)-X'(s),b(s,X_{s})-b(s,X'_{s})\Big\rangle_{\mathbb{H}}ds\\
		&\hspace{3cm}+{}_{\mathbb{V}}\Big\langle X(s)-X'(s), A(s,X(s))-A(s,X'(s))\Big\rangle_{\mathbb{V}^{*}}ds\\
		&\hspace{3cm}+\int_{U}\Big\langle X(s)-X'(s),f(s,X_{s},u)-f(s,X'_{s},u)\Big\rangle_{\mathbb{H}}\widetilde{N}(ds,du)\\
		&\hspace{3cm}-\Big\langle X(s)-X'(s),d\eta(s)-d\eta'(s)\Big\rangle_{\mathbb{H}}\Bigg]\\
		& +\int_{0}^{t}e^{-\beta s}\Bigg[\Big\|\sigma(s,X_{s})-\sigma(s,X'_{s})\Big\|_{L_{2}^{0}}^{2}ds-\beta\|X(s)-X'(s)\|_{\mathbb{H}}^{2}ds\\
		&\hspace{3cm}+\int_{U}\Big\|f(s,X_{s},u)-f(s,X'_{s},u)\Big\|_{\mathbb{H}}^{2}\nu(du)ds\\
		&\hspace{3cm}+\int_{U}\Big\|f(s,X_{s},u)-f(s,X'_{s},u)\Big\|_{\mathbb{H}}^{2}\widetilde{N}(ds,du)\Bigg].
	\end{align*}
	Set $$\overline{\tau}_{R}:=\inf\Big\{t\geq 0;\;\|X(t)\|_{\mathbb{H}}\vee\|X'(t)\|_{\mathbb{H}}>R\Big\}.$$ 
	By Lemma \ref{lem2}, the Burkh\"older inequality, and conditions (H1)-(H3), we have
	\begin{align*}
		\mathbb{E}\sup_{-h\leq s\leq T}\|X(s\wedge\overline{\tau}_{R})-X'(s\wedge\overline{\tau}_{R})\|_{\mathbb{H}}^{2}
		\leq C_{T}\int_{0}^{T}G_{R}\Big(s,\mathbb{E}\sup_{-h\leq r\leq s\wedge\overline{\tau}_{R}}\|X(r\wedge\overline{\tau}_{R})-X'(r\wedge\overline{\tau}_{R})\|_{\mathbb{H}}^{2}\Big)ds.
	\end{align*}
	Condition (H3) gives
	\begin{align*}
		\mathbb{E}\sup_{-h\leq s\leq T}\|X(s\wedge\overline{\tau}_{R})-X'(s\wedge\overline{\tau}_{R})\|_{\mathbb{H}}^{2}=0\quad\text{and}\quad
		\mathbb{E}\int_{0}^{T}\|X(s\wedge\overline{\tau}_{R})-X'(s\wedge\overline{\tau}_{R})\|_{\mathbb{V}}^{2}ds=0,
	\end{align*}
	which completes the proof of uniqueness 
	by letting $R\rightarrow\infty$.
	
	Finally, applying the energy equality to $\|X(t)-X(s)\|^2_{\mathbb{H}}$ yields
	\begin{align*}
		&\mathbb{E}\|X(t)-X(s)\|^2_{\mathbb{H}}\\
		&=2\mathbb{E}\int_s^t\Big\langle X(r)-X(s),b(r,X_{r})dr\Big\rangle_{\mathbb{H}}
		+2\mathbb{E}\int_s^t{}_{\mathbb{V}}\Big\langle X(r)-X(s),A(r,X(r))\Big\rangle_{\mathbb{V}^{*}}dr\\
		&\quad-2\mathbb{E}\int_s^t\Big\langle X(r)-X(s),d\eta(r)\Big\rangle_{\mathbb{H}}+\mathbb{E}\left[\int_{s}^{t}\|\sigma(r,X_{r})\|^{2}_{L_{2}^{0}}dr+\int_{U}\|f(r,X_{r},u)\|_{\mathbb{H}}^{2}\nu(du)\right]dr\\
		&\leq C\mathbb{E}\int_s^t\|X(r)-X(s)\|_{\mathbb{H}}^2dr+C(t-s)+\mathbb{E}\int_s^t\left[\varphi(X(r))-\varphi(X(s))\right]dr\\
		&\quad+\mathbb{E}\int_{s}^{t}\left[\|b(r,X_r)\|_{\mathbb{H}}^{2}+\|\sigma(r,X_{r})\|^{2}_{L_{2}^{0}}dr+\int_{U}\|f(r,X_{r},u)\|_{\mathbb{H}}^{2}\nu(du)\right]dr.
	\end{align*}
	Thus by applying donditions (H2) and (H4), together with Eqn. \eqref{4thmom}, 
	\begin{align*}
		\mathbb{E}\|X(t)-X(s)\|^4_{\mathbb{H}}
		&\leq C(t-s)^2+C\mathbb{E}\left(\int_s^t\Big[1+\|X(r)\|_{\mathbb{H}}^{l}+\|X(s)\|_{\mathbb{H}}^l\Big]dr\right)^2\\
		&\quad+C(t-s)\int_{s}^{t}H\Big(r, \mathbb{E}\sup_{r'\in[-h,T]}\|X(r')\|_{\mathbb{H}}^4\Big)dr\\
		&\leq C(t-s)^2+C(t-s))\int_{s}^{t}H(r,c_T)dr,
	\end{align*}
	where $c_T:=\mathbb{E}\sup_{t\in[-h,T]}\|X(t)\|_{\mathbb{H}}^4$.
\end{proof}			

\section*{Acknowledgments}
		The research of Jing Wu (Corresponding author) was supported by NSFCs (No. 12071493, 11871484). 


\bibliography{bib-ms}

\end{document}